\documentclass{article}
\usepackage{amssymb,amsmath,relsize}
\title{Graph reductions, binary rank, and pivots in gene assembly}
\author{Nathan Pflueger \footnote{\textit{E-mail address:} pflueger@math.harvard.edu}\\
Harvard University}

\usepackage[all]{xy}

\usepackage{amsmath}
\usepackage{amsthm}
\usepackage{amssymb}
\usepackage{delarray}
\usepackage[dvips]{graphics}
\usepackage{epsfig}
\usepackage{fullpage}

\newtheorem{theorem}{Theorem}[section] 
\newtheorem{lemma}[theorem]{Lemma} 
\newtheorem{proposition}[theorem]{Proposition}
\newtheorem{corollary}[theorem]{Corollary}
\newtheorem{definition}[theorem]{Definition}

\theoremstyle{remark}
\newtheorem{example}[theorem]{Example}
\newtheorem{problem}{Problem}

\newcommand{\rk}[1]{\textrm{rank}(#1)}
\newcommand{\rkg}[1]{\textrm{rank}_G(#1)}

\newcommand{\mat}[4]{ \begin{bmatrix} #1 & #2 \\ #3 & #4 \end{bmatrix} }
\newcommand{\smat}[4]{\bigl[ \begin{smallmatrix} #1 & #2 \\ #3 & #4 \end{smallmatrix} \bigr]}
\newcommand{\col}[2]{\left[ \begin{smallmatrix} #1 \\ #2 \end{smallmatrix} \right] }
\newcommand{\row}[2]{ \left[ \begin{smallmatrix} #1 & #2 \end{smallmatrix} \right]}
\newcommand{\gpr}{\textrm{gpr}}
\newcommand{\gdr}{\textrm{gdr}}
\newcommand{\gnr}{\textrm{gnr}}

\newcommand{\cE}{\mathcal{E}}
\newcommand{\cV}{\mathcal{V}}
\newcommand{\cW}{\mathcal{W}}
\newcommand{\cR}{\mathcal{R}}
\newcommand{\cP}{\mathcal{P}}
\newcommand{\perpE}{{\perp \mathcal{E}}}

\begin{document}

\maketitle

\begin{abstract}
We describe a graph reduction operation, generalizing three graph reduction operations related to gene assembly in ciliates. The graph formalization of gene assembly considers three reduction rules, called the positive rule, double rule, and negative rule, each of which removes one or two vertices from a graph. The graph reductions we define consist precisely of all compositions of these rules. We study graph reductions in terms of the adjacency matrix of a graph over the finite field $\textbf{F}_2$, and show that they are path invariant, in the sense that the result of a sequence of graph reductions depends only on the vertices removed. The binary rank of a graph is the rank of its adjacency matrix over $\textbf{F}_2$. We show that the binary rank of a graph determines how many times the negative rule is applied in any sequence of positive, double, and negative rules reducing the graph to the empty graph, resolving two open problems posed by Harju, Li, and Petre. We also demonstrate the close relation between graph reductions and the matrix pivot operation, both of which can be studied in terms of the poset of subsets of vertices of a graph that can be removed by a graph reduction.
\end{abstract}

\begin{center}
\end{center}


\section{Introduction}

This paper considers a graph reduction process formalizing gene assembly in strichotrichous ciliates. We briefly survey this background before describing the combinatorial formalization. The biological background is not necessary elsewhere in the paper.

\subsection{Ciliates and Gene Assembly}

Strichotrichous ciliates are ancient unicellular eukaryotes possessing two distinct types of cell nuclei, called the macronucleus and the microncleus. The macronucleus is the somatic nucleus, while the micronucleus is a germline nucleus that is used to transmit genes to offspring during sexual reproduction. Genes in the micronucleus are located on long molecules consisting of coding blocks separated by non-coding material. These coding blocks must be assembled into their ``orthodox order'' during reproduction. In the micronucleus, however, the blocks may be shuffled, and some may be inverted. The necessary data to assemble these blocks into the orthodox order are encoded in short nucleotide sequences called pointers, located at each end of each coding block. In effect, the coding blocks may be regarded as nodes in a doubly linked list, with the pointers at the ends of each block indicating which block precedes it and which block follows it in the orthodox order. The process of reading these pointers and assembling the blocks into the orthodox order is called the gene assembly process, and it is an example of what could be considered computation in living cells. Background on strichotrichous ciliates and the gene assembly process may be found in \cite{jahn} and \cite{prescott}. A thorough treatment of the gene assembly process and its various formalizations can be found in \cite{monograph}.

The formalization we consider comes from an intramolecular model for gene assembly, which is described in \cite{ehrenfeucht3} and \cite{prescott1}. Several mathematical formalizations for the intramolecular model are described in \cite{langille}. The most straightforward formalization makes use of signed double-occurence strings: the sequence of pointers is described by a string in which each letter occurs exactly twice, and each letter is given a sign to indicate whether or not it is inverted. Such strings are also called \textit{legal strings}. This formalization is studied further in \cite{ehrenfeucht} and \cite{harju0}. The formalization which we consider uses signed graphs, which can be obtained from legal strings as follows: the vertex set is the set of letters in the string; two vertices are connected by an edge if the corresponding letters ``interlock'' in the string (i.e. they appear in the patter $abab$, rather than $aabb$ or $abba$); a vertex is assigned the sign $+$ if it appears both inverted and non-inverted in the string, and $-$ otherwise.  Although it may appear that some information is lost in converting strings to graphs, it is demonstrated in \cite{ehrenfeucht0} that no essential information is lost, in the sense that assembly strategies in one formalization correspond to assembly strategies in the other. It is not the case, however, that all signed graphs arise from legal strings. Further discussion of the legal string and signed graph formalizations, and the relation between them, can be found in \cite{ehrenfeucht2}.

The intramolecular model postulates that gene assembly is achieved by applying a sequence of three basic molecular operations, denoted $LD, HI$, and $DLAD$. These correspond, in the legal string and signed graph formalizations, to combinatorial operations called the \textit{negative rule}, \textit{positive rule} and \textit{double rule}. In the graph formalization, each rule shrinks the vertex set of the graph by one or two vertices, and reconfigures the edges between the remaining vertices. The gene assembly process is complete when the graph has been reduced to the empty graph. These three combinatorial operations are the basis of the graph reductions which we study in this paper.

\subsection{Graph Reductions}

We shall refer to the graph formalization of the three molecular operations as \textit{combinatorial reduction rules}, and compositions of them will be called \textit{combinatorial graph reductions}, or simply \textit{graph reductions}. A \textit{successful} graph reduction is a reduction of a graph to the empty graph.

The basic problems about the graph formalization of gene assembly concern understanding the different sequences of combinatorial reductions rules in a graph that produce a successful reduction. In particular, one wishes to understand how to measure the complexity of a given signed graph from the standpoint of combinatorial graph reduction. Several measures of complexity are proposed and analyzed in \cite{harju1}. In particular, one can ask if a given graph can be reduced to the empty graph using only some subset of the three operations; a classification is given in \cite{harju} for those graphs which can be reduced without the positive rule, and those which can be reduced without the double rule; this paper completes that classification by classifying the graphs which can be reduced without the negative rule (Section \ref{nullitySection}). An active topic recently concerns the \textit{parallel complexity} of a signed graph. The parallel complexity of a signed graph is the number of steps needed to reduce it to the empty graph, if it is permitted to perform a set of operations simultaneously if and only if they could be applied in any order with the same result. Parellel complexity is studied for various families of graphs in \cite{harju} and \cite{harju2}, and the computational problem of determining a graph's parallel complexity is considered in \cite{computation} and \cite{alhazov}. It is not known whether parallel complexity can be computed in polynomial time. Surprisingly, no nontrivial bounds are known for the parallel complexity of a graph with a given number of vertices. It is not even known whether parallel comlexity is unbounded for general graphs; it is conjectured in \cite{harju} that in fact parallel complexity is bounded by a constant for all graphs.

This paper demonstrates that all of these questions can be formulated in linear algebraic terms by considering the adjacency matrix of the graph over $\textbf{F}_2$, where the sign of a vertex is encoded by regarding positive vertices as having loops. This idea has also been pursued in \cite{brijder}. We generalize a result from \cite{harju} by demonstrating that the combinatorial reduction rules on signed graphs are path-invariant, in the sense that the result of removing a given subset of vertices does not depend on the particular operations used to remove them (Theorem \ref{pathInvariance}); we also provide an algebraic criterion determining whether a given set of vertices can be removed (Proposition \ref{matrixReducibility}). We prove that the number of times the negative rule is applied in a reduction to the empty graph is determined by the rank of the adjacency matrix (Theorem \ref{ngrNullity}), thus classifying the graphs which can be reduced without this rule and resolving two problems posed in \cite{harju}.

The methods in this paper also suggest another way to encode the data of the possible reductions of a signed graph: by a poset. In particular, the set of subsets of vertices which can be removed without the negative rule (equivalently, as we demonstrate, the subsets whose induced subgraph has invertible adjacency matrix) forms a poset that we call the \textit{pivotal poset}. This poset completely determines the original graph (Theorem \ref{posetDeterminesGraph}), and naturally encodes the sequences of steps that apply in parallel, thus suggesting a new approach to the study of parallel complexity.

The notion of a graph pivot, first considered in the context of gene assembly in \cite{brijder}, is closely related to the methods of this paper. In particular, we characterize the pivot operation in terms of the pivotal poset. As an application of these ideas, we describe and solve in Section \ref{reverseReductions} what might be called the inverse problem for reductions of signed graphs: given a graph, which graphs can be reduced to it using combinatorial reduction rules?

Some of our results, in particular regarding path invariance of graph reductions, have been proved in weaker forms in \cite{brijder}, by means of the matrix pivot operation on the adjacency matrix. In effect, their work concerns what we refer to as \textit{nonsingular reduction}, which can be characterized as those graph reductions which do not use the negative rule. In Section \ref{pivots}, we show this connection, and give a simple characterization of pivots of graphs in terms of the reducibility poset. Our method generalizes some results from \cite{brijder}, and provides short proofs for others. We discuss the implications of the pivot operation for the reducibility poset. As a special case, we consider the \textit{retrograph} of a graph, which can be defined by taking the inverse of the adjacency matrix, when it exists.

We begin by describing the combinatorial reduction rules in Section \ref{combinatorial}. We generalize the reduction rules in \ref{algebraic} using linear algebra over $\textbf{F}_2$, and prove our path-invariance result. We demonstrate in Section \ref{combMinimal} that the combinatorial reduction rules are simply the minimal graph reductions, and also give our results on the number of applications of the negative rule in a successful reduction. In Section \ref{pivots} we relate graph reductions to the matrix pivot operation, and describe the relation between pivots, the pivotal poset, and the graph reduction inverse problem.

Throughout the paper, we shall use the word \textit{graph} to refer to a simple graph with loops (i.e. there is at most one edge between any two vertices, and vertices may have edges to themselves), and the \textit{adjacency matrix} of a graph will always be understood to have coefficients in $\textbf{F}_2$. When we refer to a \textit{signed graph}, we mean a simple graph without loops together with an assigned sign ($+$ or $-$) for each vertex. These two notions are equivalent in the sense that the sign $+$ may be understood to indicate that  the vertex has a loop edge.

\section{Combinatorial Graph Reductions}\label{combinatorial}
We begin by describing the graph reduction operations formalizing the three molecular operations. We shall refer to these reductions as \emph{combinatorial graph reductions}, in order to distinguish them from the definition of graph reductions that we give in the next section. In Section \ref{combMinimal} we shall demonstrate that these two notions of graph reduction coincide. These reductions have been considered on signed graphs until now, so we present this viewpoint first. We then describe how these rules can be equivalently formulated on simple graphs with loops (which we shall call, simply, \textit{graphs}), and demonstrate that this leads to simple formulas for the reduction rules in terms of the adjacency matrix.

\subsection{On signed graphs}
The three molecular operations postulated by the intramolecular model are HI, DLAD, and LD. Each has a corresponding rule on signed graphs, defined as follows.

\begin{definition}
A \emph{signed graph} $G = (V,E,\sigma)$ is a simple graph on vertices $V = \{v_1, v_2, \dots, v_n\}$, with edges $E$, such that each vertex is given a sign by $\sigma: V \rightarrow \{+,-\}$.
\end{definition}

Let $N_G(v)$ denote the neighborhood of $v$ in $G$ (not including $v$ itself). By \textit{complementing} an edge $(v_1,v_2)$, where $v_1,v_2$ are vertices, we mean adding an edge between $v_1$ and $v_2$ if one is not present, and removing the edge between $v_1$ and $v_2$ if one is present.

\begin{definition}
The three combinatorial reduction rules on signed graphs are as follows.
\begin{itemize}
\item $\emph{gpr}_v$, the \emph{graph positive rule} applies if and only if $\sigma(v)=+$. It removes $v$ from the graph, all edges among two vertices in $N_G(v)$ are complemented, and the signs of all vertices in $N_G(v)$ are inverted.
\item $\emph{gdr}_{v_1,v_2}$, the \emph{graph double rule} applies if and only if $\sigma(v_1)=\sigma(v_2)=-$ and $(v_1,v_2) \in E$. It removes $v_1$ and $v_2$ from the graph, and complements all edges $(x,y)$ such that one of $x$ or $y$ lies in $N_G(v_1)$ and the other in $N_G(v_2)$, but such that not both $v_1$ and $v_2$ lie in $N_G(v_1) \cap N_G(v_2)$. Signs are unaffected.
\item $\emph{gnr}_v$, the \emph{graph negative rule} applies if and only if $\sigma(v)=-$ and $v$ is isolated (has no neighbors). It removes $v$, and does not affect the rest of the graph.
\end{itemize}
A \emph{combinatorial reduction strategy} is a sequence $(\gamma_1, \gamma_2, \dots, \gamma_n)$ of one or more combinatorial reduction rules. A combinatorial reduction strategy is called \emph{applicable} if for $i =1,2,\dots,n$, the rule $\gamma_i$ applies to $\gamma_{i-1} \circ \dots \circ \gamma_2 \circ \gamma_1 (G)$. The \emph{domain} of a reduction strategy is the set of vertices removed by the strategy. A reduction strategy is called \emph{successful} if it is applicable and its domain is all of $V$. The composition $\gamma_n \circ \dots \circ \gamma_2 \circ \gamma_1$ of the rules of a combinatorial reduction strategy is called a \emph{combinatorial reduction}.
\end{definition}

See \cite{ehrenfeucht3}, \cite{prescott1}, \cite{ehrenfeucht0} and the monograph \cite{monograph}, for discussion of these three rules and their relation to the postulated molecular operations HI, DLAD, and LD. We also point out that if we consider graphs with only negative vertices, the double rule $\gdr$ is identical to the \textit{rank two reduction} rule as considered in \cite{ggt}.

An example of a successful reduction strategy of a signed graph, demonstrating the three rules, is shown in Figure 1. There are other successful reduction strategies for this graph (for example, vertex $v_3$ can be removed first, using the positive rule, see Figure \ref{fig:pathInvarianceExample}). All diagrams have been created using the gene assembly simulator \cite{simulator}.

\begin{figure}[h]
\begin{center}
\noindent\begin{tabular}{c|c}
$G$& $\gdr_{v_1,v_2}(G)$\\ &\\
\includegraphics[scale=0.3]{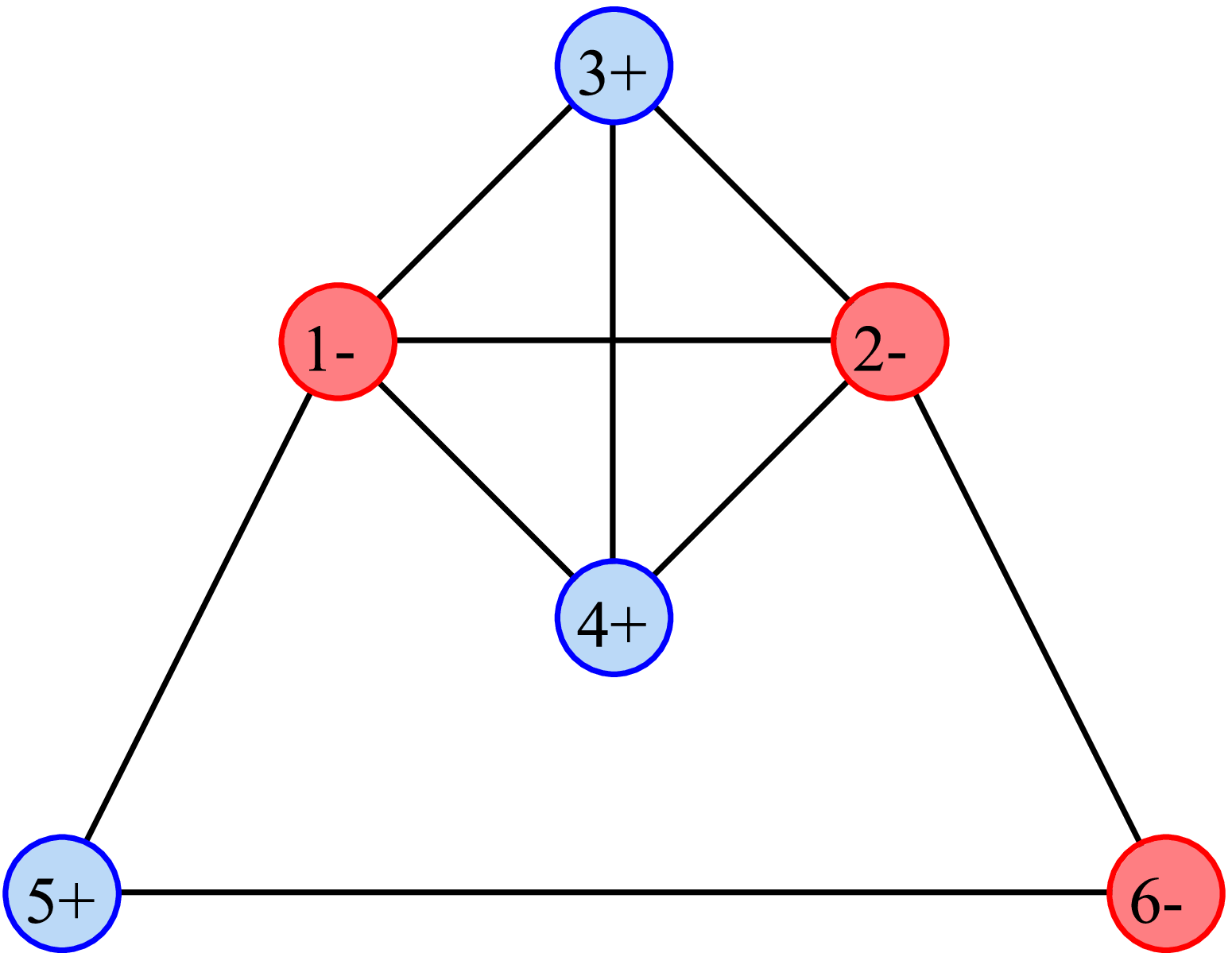}&
\includegraphics[scale=0.3]{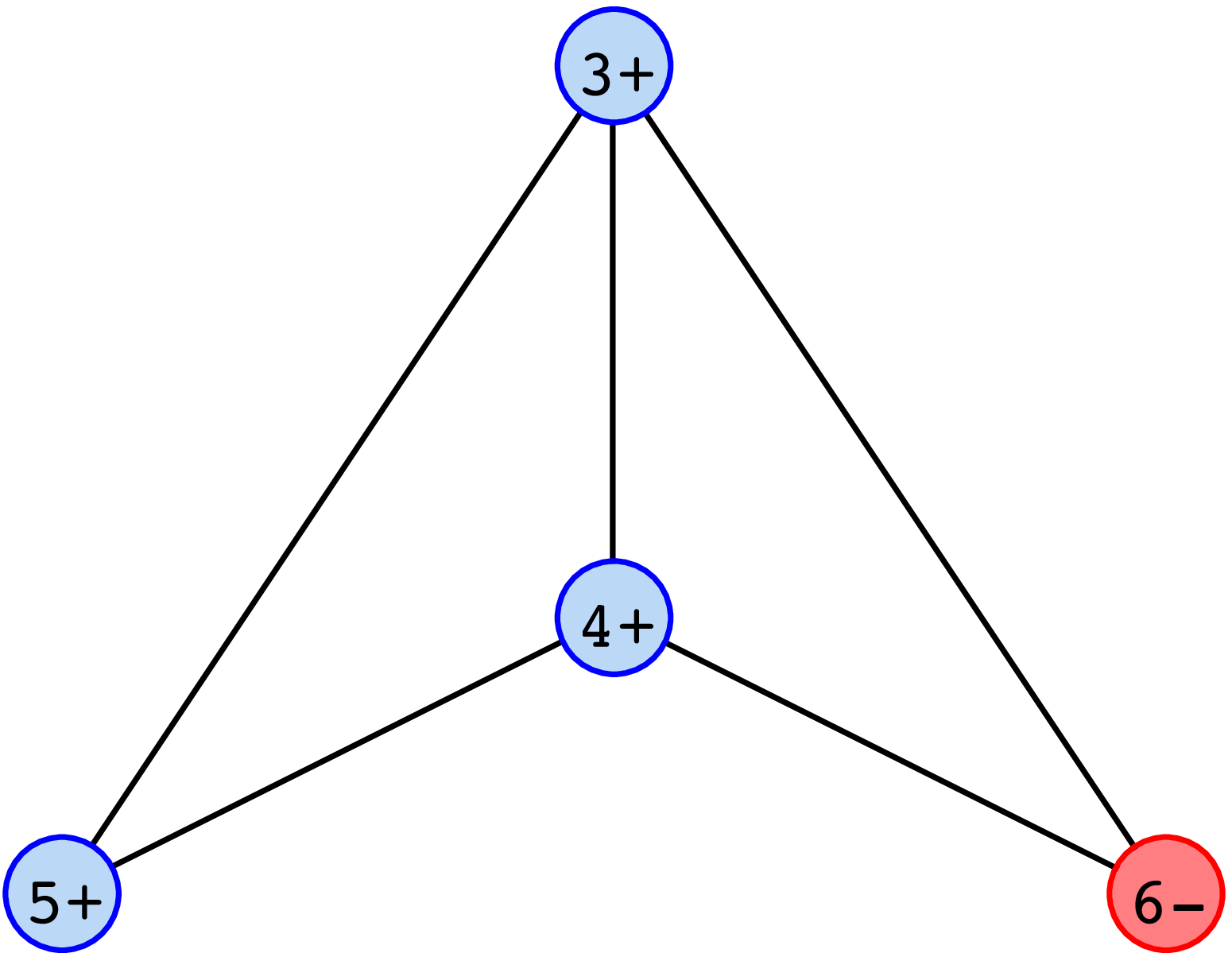}\\
\hline
$(\gpr_{v_3} \circ \gdr_{v_1,v_2})(G)$ & $(\gnr_{v_4} \circ \gpr_{v_3} \circ \gdr_{v_1,v_2})(G)$\\ &\\
\includegraphics[scale=0.3]{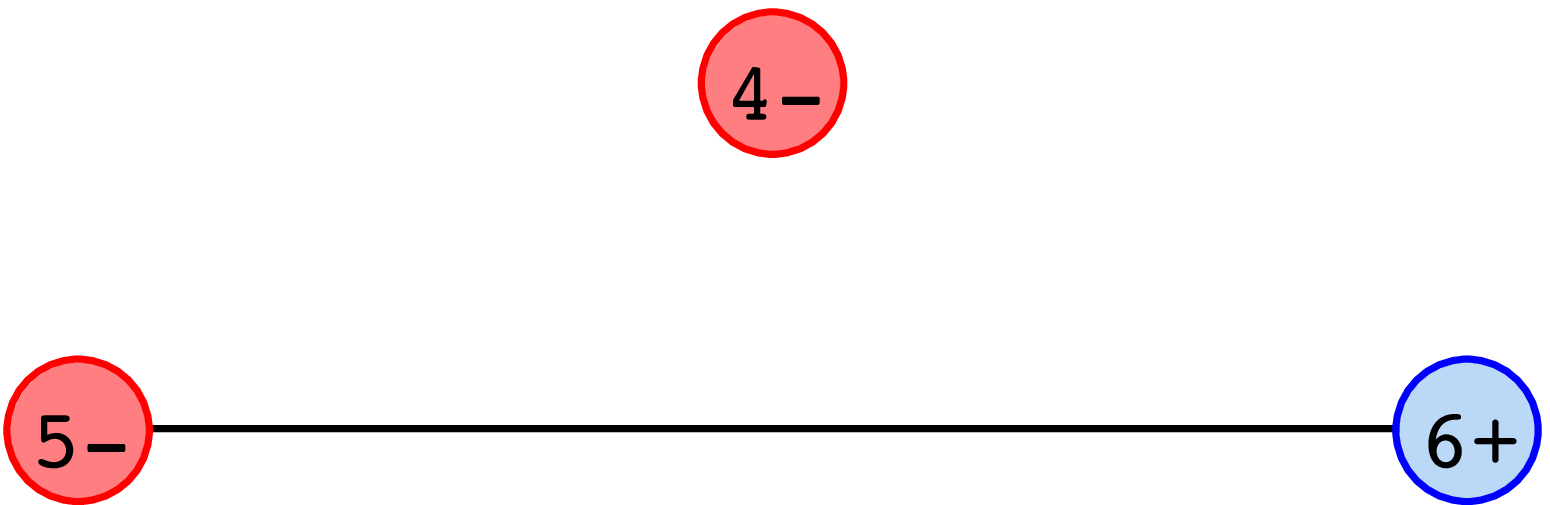}&
\includegraphics[scale=0.3]{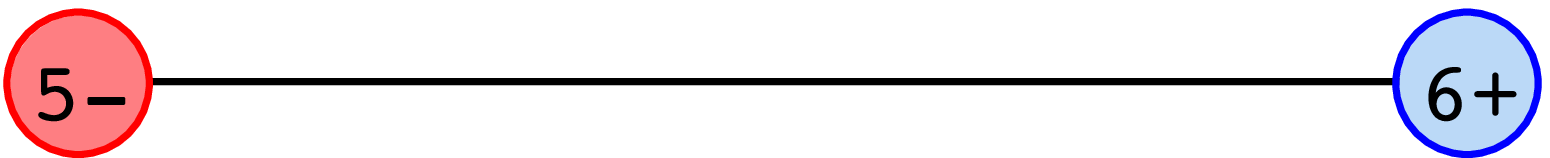}\\
\hline
$(\gpr_{v_6} \circ \gnr_{v_4} \circ \gpr_{v_3} \circ \gdr_{v_1,v_2})(G)$ & $(\gpr_{v_5} \circ \gpr_{v_6} \circ \gnr_{v_4} \circ \gpr_{v_3} \circ \gdr_{v_1,v_2})(G)$\\ &\\
\includegraphics[scale=0.3]{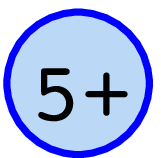}& $\emptyset$
\end{tabular}\\
\end{center}
\caption{A successful combinatorial reduction strategy.}
\end{figure}

Observe that in any nonempty signed graph, at least one of the combinatorial reduction rules is applicable, and the vertex set shrinks whenever any rule is applied. Thus every signed graph has some successful combinatorial reduction strategy. We consider the set of all successful reduction strategies of an arbitrary signed graph. In Section \ref{redDefinitions} we will obtain a simple algebraic description of those vertex sets which can be removed by some combinatorial reduction strategy. The first step in this direction is to reinterpret signed graphs in a way that will allow them to be studied algebraically.

\subsection{On simple graphs with loops}

A \textit{simple graph with loops} is a graph without multiple edges, but where a vertex may have an edge to itself. There is a bijection between signed graphs and simple graphs with loops, by regarding positive vertices to be vertices with loops and negative vertices as vertices without loops. We will thus use the two viewpoints interchangeably. In this paper, we shall use the word \textit{graph} to mean simple graph with loops.

The main advantage of this second viewpoint is that a signed graph with loops can be described by an adjacency matrix, where the diagonal of the matrix indicates which vertices have loops. If we regard the entries of this matrix as lying in the finite field $\textbf{F}_2$, then the three combinatorial reduction rules are easy to state in terms of the adjacency matrix. If we order the vertices of the graph so that the domain of the reduction comes first, then we may express the three operations in terms of block matrices, as follows. It is important to recall that the entries are in $\textbf{F}_2$, not $\textbf{R}$. The submatrix $Q$ is any $1 \times (n-1)$ matrix in the first line and any $2 \times (n-1)$ matrix in the second line. In the third line, $\textbf{0}$ denotes the $1 \times (n-1)$ vector of all $0$s.

\begin{eqnarray}
\label{combMatrix1}
 \gpr_v: & \smat{1}{Q}{Q^T}{R} &\mapsto R - Q^TQ\\
 \label{combMatrix2}
 \gdr_{v_1,v_2}: & \mat{\begin{smallmatrix} 0 & 1 \\ 1 & 0 \end{smallmatrix}}{Q}{Q^T}{R} &\mapsto R - Q^T\smat{0}{1}{1}{0} Q\\
\label{combMatrix3} 
\gnr_v: & \smat{0}{\textbf{0}}{\textbf{0}^T}{R} &\mapsto R
\end{eqnarray}

We also point out here that the positive rule and double rule each reduce the rank of the adjacency matrix by precisely the number of vertices removed. This is why Godsil and Royle \cite{ggt} refer to the double rule as a rank two reduction. This fact can be seen by realizing both rules as a sequence of row reduction operations, and then a restriction to a principle submatrix. We omit the details here since this result will also follow from Corollary \ref{rankCorollary} after discussing general graph reductions.

\section{Graph reductions in general}\label{algebraic}
We give in this section a linear algebraic description of the combinatorial graph reductions described above. This description will allow us to prove path invariance for graph reductions, and also characterize the number of times the negative rule $\gnr$ is used in a given reduction, resolving two open problems from \cite{harju}. We will also obtain formulas to compute all edge relations of a graph after reduction in terms of ranks of submatrices of the adjacency matrix. These formulas generalize the determintant formulas given in \cite{brijder}, which apply only in the absence of the negative rule. The reductions we define here are closely related to the pivot operation on matrices defined in \cite{geelen} and studied in \cite{brijder}, which we consider in Section \ref{pivots}. We observe that all of our work in this section is easily generalized to directed graphs by considering asymmetric adjacency matrices, but we consider only symmetric adjacency matrices in order to simplify notation.

\subsection{Preliminaries}

We begin with an intrinsic definition of graph reductions. In Section \ref{matrices} we will interpret this definition using matrices in block form. Suppose $G$ is a graph, on vertices $V = \{ v_1, v_2, \dots, v_n \}$, with edges $E$. We shall denote by $\mathcal{V}$ the $n$-dimensional vector space over $\textbf{F}_2$ (the finite field with two elements) with basis $V$. For any subset $W \subset V$, $\langle W \rangle$ will denote the span in $\cV$ of the vertices in $W$ (in particular, $\langle V \rangle = \mathcal{V}$). We shall denote by $\mathcal{E}$ a symmetric bilinear form on $\mathcal{V}$ defined on basis vectors as follows.

\begin{equation}
\mathcal{E}(v_i, v_j) = 
\begin{cases}
 1 & (v_i,v_j) \in E \\
 0 & (v_i,v_j) \not\in E
\end{cases}
\end{equation}

Observe that the bilinear form $\cE$ is given by the adjacency matrix $A$ of the graph, in the sense that $\cE(v_1,v_2) = v_1^T A v_2$.

This form is defined on all of $\mathcal{V}\times \mathcal{V}$ by bilinearity. Recall that we permit $G$ to have loops, and $\mathcal{E}(v_i,v_i) = 1$ if and only if vertex $v_i$ has a loop. We shall describe the results of reductions of $G$ by specifying different bilinear forms, using the following notation.

\begin{definition}
 For any set of vertices $W$, and any symmetric bilinear form $\mathcal{F}$ defined on $\langle W \rangle$, we denote by $\mathcal{G}(W,\mathcal{F})$ the graph on vertices $W$ with edges $\{ (w_i, w_j):\ \mathcal{F}(w_i,w_j) = 1 \}$.
\end{definition}

For example, the graph $G$ can be denoted $\mathcal{G}(V,\mathcal{E})$, and for any subset $W \subset V$, the graph $\mathcal{G}(W,\mathcal{E})$ is the induced subgraph of $G$ on vertices $W$. Observe that, in the above definition, $\mathcal{F}$ may be a form on a larger vector space than $\langle W \rangle$, as in the case of induced subgraphs, although only its restriction to $\langle W \rangle$ is relevant. Graph reductions will be defined by modifying the bilinear form $\mathcal{E}$ in a manner that ``forgets'' the removed vertices in a particular way. Before giving a precise definition, we shall informally motivate the idea behind it.

\subsection{Motivation}\label{motivation}

We begin with a very simple principle: if $W \subset V$ is a subset of vertices such that the induced subgraph of $G$ on vertices $W$ is not connected to the rest of the graph, then the graph reduction removing the vertices $W$ must simply be the induced subgraph on the remaining vertices. Our approach is to define the graph reduction for a set of vertices $W$ by first modifying the graph in such a way that the vertices $W$ become disconnected. This modification is naturally expressed using linear algebra. The bilinear form $\mathcal{E}$ allows not just elements of the set $V$, but in fact all elements of the vector space $\mathcal{V}$, to be regarded as vertices of a graph. By using a different basis for the vector space, a different graph on the same number of vertices is obtained, and graph reductions can be defined by performing a change of basis in a specific way. We first illustrate this idea with an example.

\begin{example}
Consider the graph on the left in Figure \ref{changeBasisExample}. In the basis $(v_1,v_2,v_2)$, the adjacency matrix is $\left[ \begin{smallmatrix} 1&1&1 \\ 1&1&0 \\ 1&0&1 \end{smallmatrix} \right]$. Consider a different basis: $(w_1,w_2,w_3) = (v_1, v_2+v_1, v_3+v_1)$. Then in this basis, the bilinear form $\mathcal{E}$ has matrix $\left[ \begin{smallmatrix} 1&0&0 \\ 0 & 0 & 1 \\ 0&1&0 \end{smallmatrix} \right]$, which is the adjacency matrix for the graph $\mathcal{G}(\{v_1,v_2+v_1,v_3+v_1\}, \cE)$. This graph is shown on the right of Figure \ref{changeBasisExample}.
\end{example}

\begin{figure}[h] 
\begin{center}
\begin{tabular}{cc}
\includegraphics[scale=0.3]{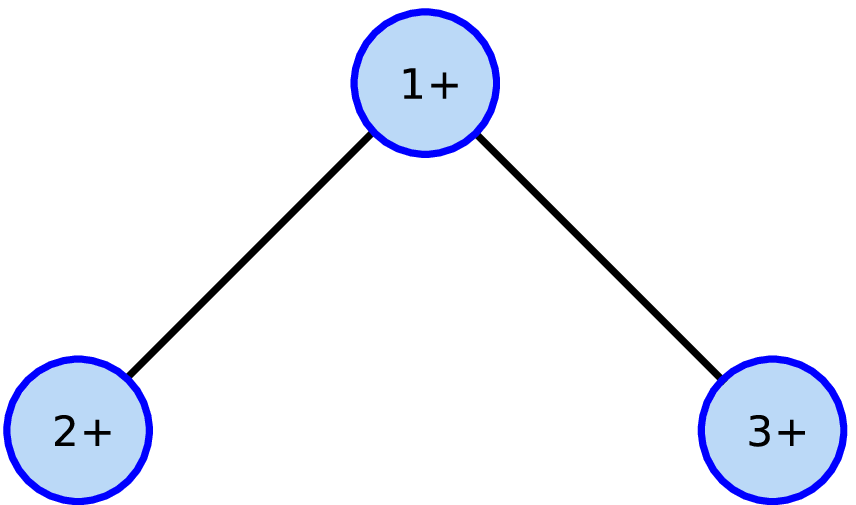} & \includegraphics[scale=0.3]{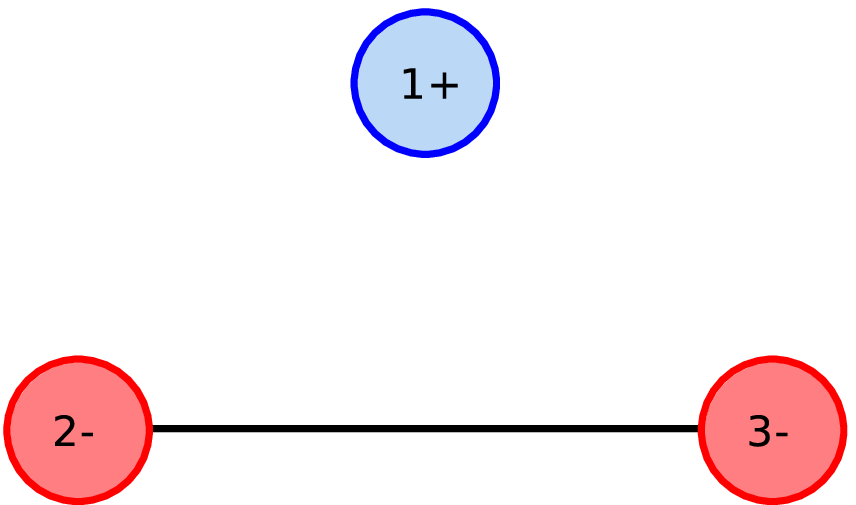}
\end{tabular}
\end{center}
\caption{A graph $G$, and the graph obtained by changing basis to $w_1 = v_1$, $w_2 = v_2+v_1$, $w_3=v_3+v_1$.}
\label{changeBasisExample}
\end{figure}

Observe that in this example, the basis is modified only by adding copies of $v_1$ to other basis vectors. In addition, the result is a graph in which the vertex $w_1$ is disconnected from the rest of the graph, and the rest of the graph is identical to the result of applying the positive rule to vertex $v_1$ in the original graph. This illustrates the principle behind our definition of graph reduction: if we wish to remove the vertices in a set $W \subset V$, we first disconnect the vertices in $W$ from the rest of the graph by changing the basis by adding linear combinations of vertices in $W$ to the vertices not in $W$, and then remove the vertices in $W$. In the example, suppose that we wish to remove the vertex $v_1$. Then we first change to the basis $(w_1,w_2,w_3)$. Then the graph reduction removing $w_1$ is the induced subgraph on vertices $(w_2,w_3) = (v_2+v_1,v_3+v_1)$. However, in the reduction process, we intend to ``forget'' the existence of the vertex $v_1$ altogether, and thus we regard $(w_2,w_3)$ as being identical to $(v_2,v_3)$ after this reduction.

We further illustrate this idea by describing the three combinatorial reduction rules in these terms.

\begin{example}[The positive rule]
Suppose that the graph $G$ on vertices $\{ v_1, \dots, v_n \}$ has adjacency matrix $\smat{1}{Q}{Q^T}{R}$ in block form. Then the bilinear form $\mathcal{E}$ becomes $\smat{1}{0}{0}{R - Q^TQ}$ in the basis $(w_1, \dots, w_n)$, where $\left[ \begin{smallmatrix} w_1 \\ \vdots \\ w_n \end{smallmatrix} \right] = \smat{1}{0}{Q^T}{I} \left[ \begin{smallmatrix} v_1 \\ \vdots \\ v_n \end{smallmatrix} \right]$. Observe that the graph corresponding to this basis is disconnected, and the induced subgraph on $\{w_2, \dots, w_n \}$ (which is congruent modulo $v_1$ to $(v_2, \dots, v_n)$) is exactly $\gpr_{v_1}(G)$.
\end{example}

\begin{example}[The double rule]
Suppose that the graph $G$ on vertices $\{ v_1, \dots, v_n \}$ has adjacency matrix $\mat{ \begin{smallmatrix} 0&1\\1&0 \end{smallmatrix} }{Q}{Q^T}{R}$ in block form. Then the bilinear form $\mathcal{E}$ becomes $\mat{ \begin{smallmatrix} 0&1\\1&0 \end{smallmatrix} }{0}{0}{R- Q^T \smat{0}{1}{1}{0} Q}$ in the basis $(w_1, \dots, w_n)$, where

\begin{equation*} \left[ \begin{smallmatrix} w_1 \\ \vdots \\ w_n \end{smallmatrix} \right] = \mat{\begin{smallmatrix} 1&0\\0&1 \end{smallmatrix} }{0}{ Q^T \smat{0}{1}{1}{0} }{I} \left[ \begin{smallmatrix} v_1 \\ \vdots \\ v_n \end{smallmatrix} \right]. \end{equation*}

Observe that the graph corresponding to this basis is disconnected, and the induced subgraph on $\{w_3, \dots, w_n \}$ (which is congruent modulo the span of  $v_1$ and $v_2$ to $(v_3, \dots, v_n)$) is exactly $\gpr_{v_1,v_2}(G)$.
\end{example}

\begin{example}[The negative rule]
By definition, the negative rule only applies to vertex $v_1$ if it is already disconnected from the rest of the graph. Hence no change of basis is necessary; the result of reducing the vertex $v_1$ is simply removing it.
\end{example}

We now make this vague notion of ``changing basis and forgetting $W$'' precise to define our notion of graph reduction.

\subsection{Definition of graph reductions}\label{redDefinitions}

Suppose that $\mathcal{W}$ is a vector subspace of $\mathcal{V}$ (for example, $\mathcal{W}$ could be the span $\langle W \rangle$ of a subset $W \subset V$). Then we wish to define the reduction of the bilinear form $\mathcal{E}$ along $\mathcal{W}$, which will be denoted $\mathcal{E}^{\mathcal{W}}$, and should correspond to ``forgetting'' the subspace $\mathcal{W}$. The easiest situation in which this can occur is if $\mathcal{E}$ can be diagonalized, in the sense that $\mathcal{V}$ can be written as a direct sum $\mathcal{V} = \mathcal{W} \oplus \mathcal{V'}$, for some other subspace $\mathcal{V'}$ (i.e. $\mathcal{V}$ is spanned by $\mathcal{W}$ and $\mathcal{V'}$, and $\mathcal{W} \cap \mathcal{V'} = \{ 0 \}$), and $\mathcal{E}(w,v) = 0$ whenever $w \in \mathcal{W}$ and $v \in \mathcal{V'}$. In this case, we define the reduction $\mathcal{E}^{\mathcal{W}}$ by projecting onto $\mathcal{V'}$ and then applying $\mathcal{E}$. The effect of this is that $\mathcal{E}^{\mathcal{W}}$ is identical to $\mathcal{E}$ on the space $\mathcal{V'}$, and is equal to $0$ when either argument comes from $\mathcal{W}$. Thus in the sense of the previous section, $\mathcal{E}^{\mathcal{W}}$ corresponds to modifying $\mathcal{E}$ so that it ``forgets'' $\mathcal{W}$. Of course, this definition will only work for certain subspaces $\mathcal{W}$, which we now define.

\begin{definition}
 If $\mathcal{W}$ is any vector subspace of $\mathcal{V}$, the \emph{$\mathcal{E}$-annihilator}, denoted $\mathcal{W}^{\perp \mathcal{E}}$, is the set $\{v \in \mathcal{V} :\ \mathcal{E}(v,w) = 0\ \forall w \in \cW \}$.
\end{definition}

\begin{definition}
 A vector subspace $\mathcal{W}$ of $\mathcal{V}$ is $\mathcal{E}$-\emph{reducible} if $\mathcal{W} + \mathcal{W}^{\perp \mathcal{E}} = \mathcal{V}$, i.e. if $\mathcal{W}$ and its $\mathcal{E}$-annihilator span $\mathcal{V}$. A subset $W$ of vertices in the graph $G$ is \emph{reducible in} $G$ if $\langle W \rangle$ is $\mathcal{E}$-reducible.
\end{definition}

We will see in Section \ref{combMinimal} that reducible sets of vertices correspond precisely to sets of vertices that can be removed by the three combinatorial reduction rules defined in Section \ref{combinatorial}. Observe that we do not require that $\mathcal{W}$ be disjoint from its $\mathcal{E}$-annihilator. Combinatorially, a set of vertices $W$ such that $\langle W \rangle$ is $\mathcal{E}$-reducible and disjoint from its $\mathcal{E}$-annihilator if and only if it can be removed from the graph using only the positive rule and the double rule (see Section \ref{nullitySection}).

Notice that $\mathcal{V}$ is not necessary a \textit{direct} sum of $\mathcal{W}$ and $\mathcal{W}^{\perp \mathcal{E}}$, since $\mathcal{W} \cap \mathcal{W}^{\perp \mathcal{E}}$ may be nonempty. However, as long as $\mathcal{W}$ and $\mathcal{W}^{\perp \mathcal{E}}$ span $\mathcal{V}$, it is possibly to find a subspace $\mathcal{V'} \subset \mathcal{W}^{\perp \mathcal{E}}$ such that $\mathcal{V}$ is the direct sum $\mathcal{W} \oplus \mathcal{V'}$. Projecting to any such subspace $V'$ and applying $\mathcal{E}$ gives the same form $\mathcal{E}^{\mathcal{W}}$ for any choice of $\mathcal{V'}$, as the following lemma demonstrates.

\begin{lemma}
Suppose that $\mathcal{W} \subset \mathcal{V}$. Then for any $v_1,v_2,v_1',v_2' \in \mathcal{W}^{\perp \mathcal{E}}$ such that $v_1 - v_1'$ and $v_2 - v_2'$ both lie in $\mathcal{W}$, $\mathcal{E}(v_1,v_2) = \mathcal{E}(v_1',v_2')$.
\end{lemma}
\begin{proof}
This is a trivial verification.
\end{proof}

This fact shows that the following is well-defined.

\begin{definition}\label{reductionFormDefn}
Suppose that $\cW \subset \cV$ is $\cE$-reducible. Then the \emph{reduction of $\cE$ along $\cW$}, denoted $\cE^\cW$, is a bilinear form on $\cV$ defined as follows. For any $v_1,v_2 \in \cV$, and any $v_1', v_2' \in \cW^{\perp \cE}$ such that $v_1 - v_1'$ and $v_2 - v_2'$ both lie in $\cW$ (such $v_1',v_2'$ exist because $\cW$ is $\cE$-reducible), define $\cE^\cW(v_1,v_2) = \cE(v_1',v_2')$.
\end{definition}

\begin{example}
Refer back to the examples in Section \ref{motivation}. In each example, take $\cW$ to be the span $\langle W \rangle$, and observe that $\cE^\cW$, when restricted to the vertices in the complement of $W$, is precisely the bilinear form corresponding to the graph obtained by removing the vertices $W$ with a combinatorial reduction rule.
\end{example}

Finally, we are able to use this approach to define graph reductions. The above example shows that this does, indeed, generalize the combinatorial reduction rules. 

\begin{definition} \label{reductionDefn}
If $W$ is a reducible set of vertices in $G$, the \emph{graph reduction of $G$ along vertices $W$} is
\begin{equation}
 \Gamma_{W}(G)= \mathcal{G}(V \backslash W, \mathcal{E}^{\langle W \rangle}).
\end{equation}
\end{definition}

This abstract definition is the easiest to manipulate to prove theorems such as path-invariance (Theorem \ref{pathInvariance}), but it is also useful to understand how this definition looks when written more explicitly using matrices. We examine this now.

\subsection{Matrix description}\label{matrices}

Suppose that $W \subset V$ is a set of vertices, and assume that $V$ is ordered with vertices of $W$ coming first. Then in this basis, $A$ can be written in block form:

\begin{equation}
 A = \mat{P}{Q}{Q^T}{R}.
\end{equation}

Thus $P$ gives the adjacency matrix for the induced subgraph on vertices $W$, $R$ for the induced subgraph on $V \backslash W$, and $Q$ describes the edges between these sets of vertices. Any vector in $\cV$ may be written in terms of this block form as $\col{w}{v}$, where $w \in \langle W \rangle, v \in \langle V \backslash W \rangle$. Whether $W$ is reducible is simply expressed in terms of this block form.

\begin{proposition}\label{matrixReducibility}
 In the notation above, the vertices $W$ are reducible if and only if the image of $Q$ is contained in the image of $P$. Equivalently, $W$ is reducible if and only if there exists a matrix $M$ such that $Q = PM$.
\end{proposition}
\begin{proof}
 Observe that $\col{w}{v} \in \cW^{\perp \cE}$ if and only if $Pw + Q v = 0$ (here by $0$ we mean an all-0 matrix). Thus $\cW$ and $\cW^{\perp \cE}$ will span $\cV$ if and only if for every $v \in \langle V \backslash W \rangle$, there exists $w \in \langle W \rangle$ such that $\col{w}{v} \in \cW^{\perpE}$, which is true if and only if the image of $Q$ is contained in the image of $P$.
\end{proof}

\begin{example} \label{minimalExample}
Suppose that $P$ is invertible. Then the set $W$ is reducible, and the matrix $M$ mentioned in Proposition \ref{matrixReducibility} must be $P^{-1} Q$. The special cases $P = \left[ 1 \right]$ and $P = \smat{0}{1}{1}{0}$ correspond to the positive rule and the double rule. On the other hand, suppose that $W = \{v_1\}$ is a single negative vertex. Then $P = \left[ 0 \right]$, and $W$ is reducible if and only if $Q = 0$; in other words, a single negative vertex is reducible if and only if it is isolated. This matches the definition of the negative rule.
\end{example}

Assume now that the vertices $W$ are indeed reducible in $G$. We shall give a formula for the adjacency matrix of $\Gamma_W (G)$.

\begin{proposition}\label{matrixGamma}
 Let $A = \smat{P}{Q}{Q^T}{R}$ be the adjacency matrix of $G$ in the basis $(v_1, v_2, \dots, v_n)$,  let $W$ be the subset of $\{ v_1,v_2, \dots, v_k \}$ of vertices, and suppose that $W$ is reducible. Let $M$ be a matrix such that $Q = PM$ (which exists since $W$ is reducible). Then the adjacency matrix of $\Gamma_W(G)$, in the basis $(v_{k+1},v_{k+2},\dots,v_n)$, is \mbox{$R-M^TPM$}. If $P$ invertible, this matrix can be written $R-Q^T P^{-1} Q$.
\end{proposition}
\begin{proof}
Let $P$ be the matrix $\smat{0}{-M}{0}{I}$. Then $P^2 = P$, i.e. $P$ is a projection. Also, the kernel of $P$ is precisely $\langle W \rangle$, and the image of $P$ is contained in $\langle W \rangle^{\perp \cE}$ (where $\cE$ is the bilinear form on $\cV$ corresponding to $A$). It follows that for any $v \in \cV$, $Pv - v \in \langle W \rangle$, and $Pv \in \langle W \rangle^{\perp \cE}$. Thus by definition \ref{reductionFormDefn}, $\cE^\cW (v_1, v_2) = \cE(Pv_1, Pv_2)$. Thus the matrix of $\cE^\cW$ is $P^T A P$. By a simply calculation, this is $\smat{0}{0}{0}{R - M^T P M}$ in block form. Restricting to $\langle V \backslash W \rangle$, we see by definition \ref{reductionDefn} that the adjacency matrix of $\Gamma_W(G)$ is $R - M^TPM$, as claimed. If $P$ is invertible, then $M = P^{-1}Q$ and thus $R - M^T P M = R - Q^T P^{-1} Q$.
\end{proof}

Observe that Proposition \ref{matrixGamma} implies that the expression $R - M^T P M$ does not depend on the choice of $M$. This fact can also established directly, using the fact that $PM$ and $M^T P$ are both independent of the choice of $M$, being $Q$ and $Q^T$, respectively.

Finally, we observe that there is a familiar linear algebraic way to compute the adjacency matrix of $\Gamma_W (G)$. If row-reduction operations are performed on $A$ until the lower-left corner becomes $0$, the result is $\mat{I}{0}{-M^T}{I} \mat{P}{Q}{Q^T}{R} = \mat{P}{Q}{0}{R - M^T P M}$. In the special case where $P$ is invertible, this reveals the relation between graph reductions and pivots, which has been considered in special cases in \cite{brijder} and which we consider in general in Section \ref{pivots}.

\subsection{Rank and determinant formulas}\label{rankFormulas}
The main insight of the algebraic approach to graph reductions is that questions about graph reductions reduce to questions about ranks of submatrices of the adjacency matrix, considered in the field $\textbf{F}_2$. We shall use the following terminology.

\begin{definition}
 If $G$ is a graph, and $W \subset V$ is a subset of the vertices of $G$, then the \emph{rank of $W$ in $G$}, $\emph{rank}_G(W)$, is the rank of the bilinear form $\cE$ restricted to $\langle W \rangle$. The \emph{nullity} of $W$ is $|W| - \emph{rank}_G(W)$. The set $W$ is called \emph{singular} if it has positive nullity, and \emph{non-singular} otherwise.
\end{definition}

Recall that the rank of the bilinear form $\cE$ restricted to $\langle W \rangle$ is equal to the rank of the submatrix of the adjacency matrix given by the considering only rows and columns corresponding to vertices in $W$.

We shall sometimes refer to the rank or nullity of $G$, by which we shall simply mean the rank or nullity of the full vertex set $V$. We observe that the rank of the graph $G$ is referred to as the \emph{binary rank} of $G$ in \cite{ggt}, in order to emphasize that the adjacency matrix is considered in $\textbf{F}_2$. We will sometimes wish to consider ranks of submatrices of the adjacency matrix, thus we use the following definition as well.

\begin{definition}
 If $G$ is a graph, and $W_1,W_2 \subset V$ are two subsets of the vertices of $G$, then $\emph{rank}_G(W_1,W_2)$ denotes the rank of the submatrix of the adjacency matrix with rows $W_1$ and columns $W_2$. More intrinsically, this is the rank of the map $\langle W_1 \rangle \rightarrow \langle W_2 \rangle^\ast$ induced by $\cE$, where $\langle W_2 \rangle^*$ is the dual vector space of $\langle W_2 \rangle$.
\end{definition}

\begin{theorem}\label{ranksOfReductions}
 If $G$ is a graph, $W$ is a reducible set of vertices in $G$, and $W_1,W_2 \subset V \backslash W$ are two sets of vertices in $\Gamma_W(G)$, then
\begin{equation}
 \emph{rank}_{\Gamma_W(G)}(W_1,W_2) = \emph{rank}_G(W \cup W_1, W \cup W_2) - \emph{rank}_G(W).
\end{equation}
\end{theorem}
\begin{proof}
 Let $\cV' \subset \cW^{\perp \cE}$ be a complementary subspace to $\cW$, in the sense that $\cV = \cW \oplus \cV'$ (as discussed in Section \ref{redDefinitions}). Then there is a projection $\pi : \cV \rightarrow \cV'$ along $\cW$ (i.e. the kernel of $\pi$ is $\cW$ and $\pi^2 = \pi$). When $\pi$ is restricted to $\langle V \backslash W \rangle$, it is an isomorphism to $\cV'$. By definition \ref{reductionFormDefn}, the bilinear form $\cE^\cW$ is given by $\cE^\cW(v_1,v_2) = \cE(\pi(v_1),\pi(v_2))$. Thus $\textrm{rank}_{\Gamma_W(G)} (W_1,W_2)$ is equal to the rank of the bilinear form $\cE$ restricted to $\pi(\langle W_1 \rangle ) \times \pi(\langle W_2 \rangle)$. Let $B$ be the matrix of $\cE$ restricted to $\pi(\langle W_1 \rangle ) \times \pi(\langle W_2 \rangle)$, and let $C$ be the matrix of $\cE$ restricted to $\langle W \rangle \times \langle W \rangle$. Then because $\cV' \subset \cW^{\perp \cE}$, matrix of $\cE$ resrticted to $\pi(\langle W_1 + W  \rangle ) \times \pi(\langle W_2 + W \rangle)$ has a ``diagonal'' block form $\smat{B}{0}{0}{C}$, hence its rank is $\textrm{rank}(B) + \textrm{rank}(C)$, which is $\textrm{rank}_{\Gamma_W(G)}(W_1,W_2) + \textrm{rank}_G(W)$. Now, since $\pi$ is projection along $\langle W \rangle$, it follows that $\pi(\langle W_1 \rangle) + \langle W \rangle = \langle W_1 \rangle + \langle W \rangle = \langle W \cup W_1 \rangle$, and similarly $\pi(\langle W_2 \rangle) + \langle W \rangle =  \langle W \cup W_2 \rangle$. Therefore $\textrm{rank}_{\Gamma_W(G)}(W_1,W_2) + \textrm{rank}_G(W) = \textrm{rank}_G(W \cup W_1, W \cup W_2)$, which gives the theorem.
 \end{proof}

\begin{corollary}\label{rankCorollary}
 $\textrm{rank}(G) = \textrm{rank}_G(W) + \textrm{rank}(\Gamma_W(G))$.
\end{corollary}
\begin{proof}
 Take $W_1 = W_2 = V \backslash W$ in the theorem.
\end{proof}

\begin{corollary}\label{rankstellvertices}
 If $W$ is a reducible vertex set in $G$ and $v,w \in V \backslash W$, then the edge $(v,w)$ is present in $\Gamma_W(G)$ if any only if
\begin{equation}
 \textrm{rank}_G(W \cup \{v\},W \cup \{w\}) > \rkg{W}.
\end{equation}
In case $W$ is nonsingular, this is equivalent to saying that $\emph{det}(A_{W \cup \{v\},W \cup \{w\}}) \neq 0$.
\end{corollary}
\begin{proof}
 The edge $(v,w)$ is present in $\Gamma_W(G)$ if and only if $\textrm{rank}_{\Gamma_W(G)}(\{v\},\{w\}) = 1$, and otherwise this rank is $0$. The result now follows immediately. 
\end{proof}

From Corollary \ref{rankstellvertices}, we see that the entire adjacency matrix of any graph reduction of $G$ can be obtained from simply knowing ranks of submatrices of the adjacency matrix. If one only considers nonsingular reductions, then it suffices to know determinants of submatrices. This leads us to consider a third abstraction for understanding reductions of simple graphs with loops: the reducibility poset.

\subsection{Path invariance and the reducibility poset}\label{pathInvarianceSection}
One of the most basic facts about graph reductions that comes to light in the algebraic formulation is the following path invariance property.

\begin{theorem}\label{pathInvariance}
Suppose $W_1, W_2$ are two disjoint sets of vertices in $G$. Then $W_2$ is reducible in $\Gamma_{W_1}(G)$ if and only if $W_1 \cup W_2$ is reducible in $G$. In this case,
\begin{equation}
 \Gamma_{W_1 \cup W_2} (G) = \Gamma_{W_2} \circ \Gamma_{W_1} (G).
\end{equation}
\end{theorem}

The proof we give here deduces the theorem easily from Theorem \ref{ranksOfReductions}. However, we point out that it is not difficult to prove the theorem directly from the definitions in Section \ref{redDefinitions}. Indeed, the result is intuitive if one regards graph reduction as ``forgetting'' vertices as described in Section \ref{motivation}, and it is simply necessary to properly formalize this intuition. However, the notation is cumbersome, so we have chosen to give the proof in terms of ranks of submatrices instead.

\begin{proof}
Observe that if $W \subset V$ is a set of vertices of $G$, then $W$ is reducible if and only if $\textrm{rank}_G(W,V) = \rkg{W}$. This follows by writing the adjacency matrix in block form $A = \smat{P}{Q}{Q^T}{R}$ (with $P$ the adjacency matrix of the induced subgraph on $W$) and observing that the column space of $\row{P}{Q}$ is equal to the column space of $P$ if and only if the column space of $Q$ is contained in the column space of $P$, which is true if and only if $W$ is reducible in $G$.

From this, we see that assuming $W_1$ is reducible in $G$, $W_2$ is reducible in $\Gamma_{W_1}(G)$ if and only if $\textrm{rank}_{\Gamma_{W_1}(G)}(W_2,V \backslash W_1) = \textrm{rank}_{\Gamma_{W_1}(G)}(W_2)$. By Theorem \ref{ranksOfReductions}, this is true if and only if $\textrm{rank}_G (W_1 \cup W_2, V) = \textrm{rank}_G (W_1 \cup W_2)$, which is true if and only if $W_1 \cup W_2$ is reducible in $G$. This establishes the first part of the theorem.

Now suppose $W_1$ and $W_1 \cup W_2$ are reducible in $G$, and $v,w \in V \backslash (W_1 \cup W_2)$. By Corollary \ref{rankstellvertices}, $(v,w)$ is an edge in $\Gamma_{W_2} \circ \Gamma_{W_1} (G)$ if and only if $\textrm{rank}_{\Gamma_{W_1}(G)} (W_2 \cup \{v\}, W_2 \cup \{w\}) > \rk{\Gamma_{W_1}(G)}$. But by Theorem \ref{ranksOfReductions}, the left side of this inequality is $\textrm{rank}_G(W_1 \cup W_2 \cup \{v\}, W_1 \cup W_2 \cup \{w\}) - \rkg{W_1}$, while the right side is $\rk{G} - \rkg{W_1}$. Thus $(v,w)$ is an edge in $\Gamma_{W_2} \circ \Gamma_{W_1} (G)$ if and only if $\textrm{rank}_{G}(W_1 \cup W_2 \cup \{v\},W_1 \cup W_2 \cup \{w\}) > \textrm{rank}_{G}(W_1 \cup W_2)$. By Corollary \ref{rankstellvertices}, this is the case if and only if $(v,w)$ is an edge in $\Gamma_{W_1 \cup W_2}(G)$. Thus $\Gamma_{W_1 \cup W_2} (G)$ and $\Gamma_{W_2} \circ \Gamma_{W_1} (G)$ have the same adjacency matrix, and thus they are the same graph.
\end{proof}

Theorem \ref{pathInvariance} demonstrates that the reducible sets of $G$ determine the reducible sets of the results of every graph reduction of $G$. In fact, if we consider the class of all graphs that can be obtained by a graph reduction from $G$, then the graph reductions from one such graph to another are in bijection with inclusions of reducible vertex sets in $G$. To be precise, if $W_1, W_2$ are two reducible vertex sets such that $W_1 \subset W_2$, then there is a unique graph reduction taking $\Gamma_{W_1}(G)$ to $\Gamma_{W_2}(G)$.

We also point out that we can express Theorem \ref{pathInvariance} in a category-theoretic way: it shows that we can define a category whose objects are simple graphs with loops, and whose morphisms are graph reductions. The fact that the composition of two graph reductions is a graph reduction makes compositions of morphisms in this category well-defined. The full subcategory of a given graph $G$ and the results of all graph reductions of $G$ is thus equivalent to the poset of reducible subsets of the vertices of $V$ (i.e. the category whose objects are these subsets, and whose morphisms are inclusion maps). Thus the category of reducible vertex sets is an important invariant of a graph.

\begin{definition}
 The \emph{reducibility poset} $\cR(G)$ of a graph $G$ is the collection of reducible subsets of vertices $S \subset V$. The \emph{reducibility poset at level $n$}, $\cR_n(G)$, is the collection of reducible subsets $S \subset V$ with nullity $n$ in $G$. The subsets $\cR_n(G) \subset \cR(G)$ will be called the \emph{levels} of the reducibility poset. The first level $\cR_0(G)$ will be called the \emph{pivotal poset} of $G$.
\end{definition}

\begin{example}
The graph from Figure \ref{changeBasisExample} has reducibility poset shown on the left of Figure \ref{fig:posets}, on page \pageref{fig:posets}.
\end{example}

Observe that if $G$ is nonsingular, then $\cR(G) = \cR_0(G)$. Also observe that all inclusions within the reducibility poset always go to higher levels. This can be proved using the results of the previous section, but it will also follow from the fact that the nullity of a reducible vertex set is precisely the number of times the negative rule must be used in a combinatorial reduction strategy removing those vertices, which we shall establish in Theorem \ref{ngrNullity}.

The primary reason that the reducibility poset is of interest is that it is possible to reconstruct the graph from the reducibility poset along with its levels; in fact, the pivotal poset suffices.

\begin{theorem}\label{posetDeterminesGraph}
 A graph $G$ is uniquely determined by $V$ and $\cR_0(G)$.
\end{theorem}
\begin{proof}
 Observe that for any $v \in V$, $v$ has a loop if and only if $\{v\} \in \cR_0(G)$. Thus the diagonal of the adjacency matrix can be recovered from the pivotal poset. Now for any $v,w \in V$, $v \neq w$, if the induced subgraph on vertices $\{v,w\}$ is $\smat{a}{b}{b}{c}$, then $\{v,w\} \in \cR_0(G)$ if and only if $ac-b^2 = 1$. Since $a$ and $c$ are diagonal entries, and $b^2 = b$ in $\textbf{F}_2$, $b$ can also be recovered from the pivotal poset. Thus the entire adjacency matrix can be obtained in this way.
\end{proof}

\begin{definition}
 A pair $(\cR_0, V)$, where $V$ is a finite set and $\cR_0$ is a collection of subsets of $V$, is called \emph{realizable} if there exists a graph $G$ on vertices $V$ with pivotal poset $\cR_0$.
\end{definition}

In principle, we can study all graph reductions by studying pivotal posets. However, this requires classifying those posets which can occur as pivotal posets. The pivot operation provides some intriguing results in this direction; we will take up that subject in Section \ref{pivots}.

\section{Combinatorial rules as minimal reductions}\label{combMinimal}
We demonstrate in this section that the general graph reductions we defined in Section \ref{algebraic} do in fact generalize the combinatorial graph reduction rules defined in Section \ref{combinatorial}, and show that the combinatorial reduction rules can be characterized simply as minimal graph reductions. We shall then give a combinatorial interpretation to the nullity of a vertex set in a graph, which will resolve two questions posed by Harju, Li, and Petre \cite{harju} about the combinatorial reduction rules. First we show that the combinatorial reduction rules are, in fact, graph reductions as defined in Section \ref{algebraic}.

\begin{lemma}
 Suppose that $W$ is the domain of a combinatorial reduction rule $\gamma$ on $G$ (i.e. $W$ is a single vertex with a loop, two adjacent vertices without loops, or an isolated vertex without a loop). Then $W$ is reducible, and $\Gamma_W$ is the same as $\gamma$, in the sense that $\gamma(G) = \Gamma_W (G)$.
\end{lemma}
\begin{proof}
 This follows by comparing formulas \ref{combMatrix1}, \ref{combMatrix2}, and \ref{combMatrix3} to Proposition \ref{matrixGamma}.
\end{proof}

We now demonstrate that the combinatorial reduction rules in fact arise naturally from the notion of graph reduction, in the sense of the following theorem.

\begin{definition}
 A nonempty subset $W$ of vertices is \emph{minimally reducible} if it is reducible, and none of its nonempty subsets are reducible. A \emph{minimal graph reduction} is a reduction $\Gamma_W$ such that $W$ is minimally reducible.
\end{definition}

\begin{theorem}
 The minimal reductions of a graph $G$ are precisely the applicable combinatorial reduction rules, $\emph{gpr}, \emph{gdr}$ and $\emph{gnr}$.
\end{theorem}
\begin{proof}
 It is easy to see that a combinatorial reduction rule is in fact minimal: the domains of $\textrm{gpr}$ and $\textrm{gnr}$ have no nonempty subsets, and both nonempty subsets of the domain of $\textrm{gdr}$ are not reducible. Thus it remains to show that if $W$ is any reducible vertex set in $G$, $W$ has a reducible subset that is the domain of a combinatorial reduction rule. If $W$ contains any positive vertices (i.e. vertices with loops), then any of these vertices is a reducible set by itself. If $W$ contains only negative vertices, and there is at least one edge between vertices of $W$, then the vertices of this edge are the domain of an applicable double rule. The only remaining case is if $W$ contains only negative vertices, and there are no edges among the vertices of $W$. Then $G$ has adjacency matrix $\smat{\textbf{0}}{Q}{Q^T}{R}$ in block form (with rows and columns corresponding to $W$ coming first). Since $W$ is reducible, the row space of $Q$ is contained in the row space of $\textbf{0}$, thus $Q = 0$, and all the vertices of $W$ are isolated, thus any one of them is the domain of an applicable negative rule. Thus in all cases, there is a reducible subset of $W$ that is the domain of an applicable combinatorial reduction rule. 
\end{proof}

\begin{corollary}
 Every graph reduction is a composition of combinatorial reduction rules.
\end{corollary}
\begin{proof}
 This follows by induction on the size of $W$.
\end{proof}

\begin{corollary}
 Combinatorial reduction strategies are path invariant, in the sense that any two strategies that remove the same set of vertices result in the same graph.
\end{corollary}
\begin{proof}
 If both strategies remove vertex set $W$, then both are equivalent to $\Gamma_W$, by Theorem \ref{pathInvariance}.
\end{proof}

\begin{example}
Figure \ref{fig:pathInvarianceExample} shows an example of two different ways to factor a graph reduction into combinatorial reduction rules, and also illustrates the path-invariance property of combinatorial reduction rules.
\end{example}

\begin{figure}[p]
\begin{center}
\begin{tabular}{c||c}
\setlength{\tabcolsep}{0.5cm}
$G$ & $G$\\
\includegraphics[scale=0.3]{example1.eps} & \includegraphics[scale=0.3]{example1.eps}\\
\\ \hline \\
$\gdr_{v_1,v_2}(G)$ & $\gpr_{v_3}(G)$\\
\includegraphics[scale=0.3]{example1a.eps} & \includegraphics[scale=0.3]{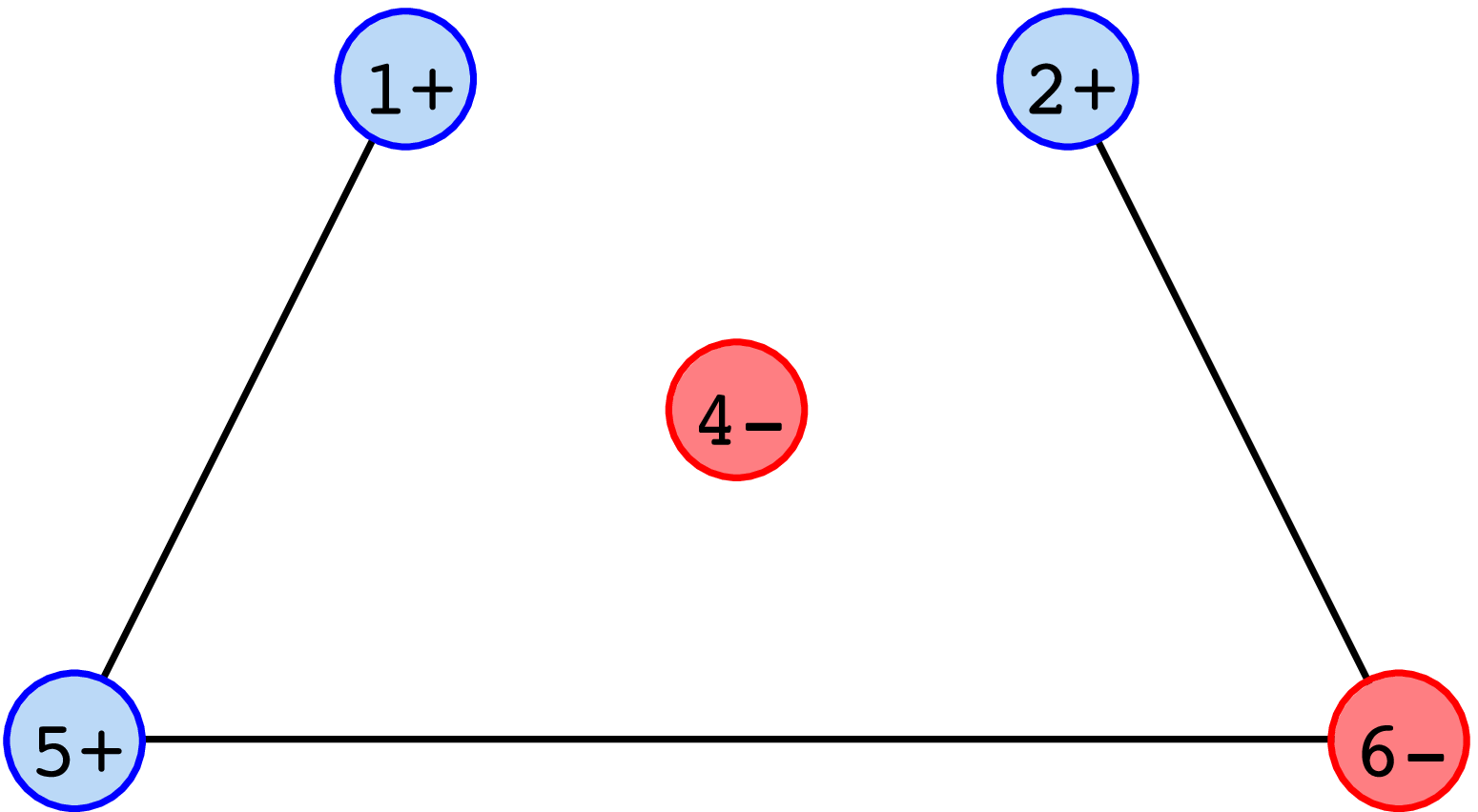}\\
\\ \hline \\
$\gpr_{v_3} \circ \gdr_{v_1,v_2}(G)$ & $\gpr_{v_1} \circ \gpr_{v_3}(G)$\\
\includegraphics[scale=0.3]{example1b.eps} & \includegraphics[scale=0.3]{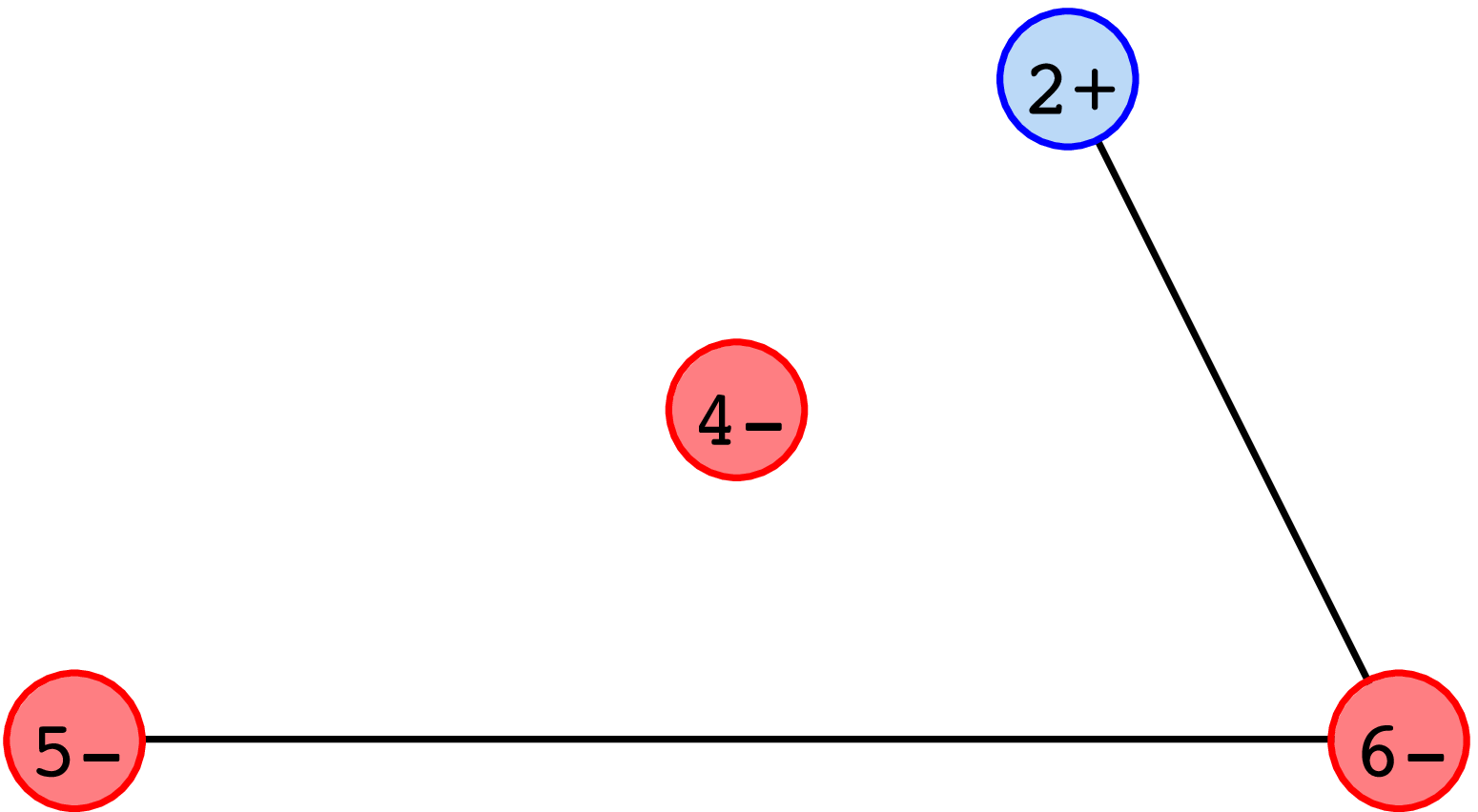}\\
\\ \hline \\
$\gnr_{v_4} \circ \gpr_{v_3} \circ \gdr_{v_1,v_2}(G)$ & $\gpr_{v_2} \circ \gpr_{v_1} \circ \gpr_{v_3}(G)$\\
\includegraphics[scale=0.3]{example1c.eps} & \includegraphics[scale=0.3]{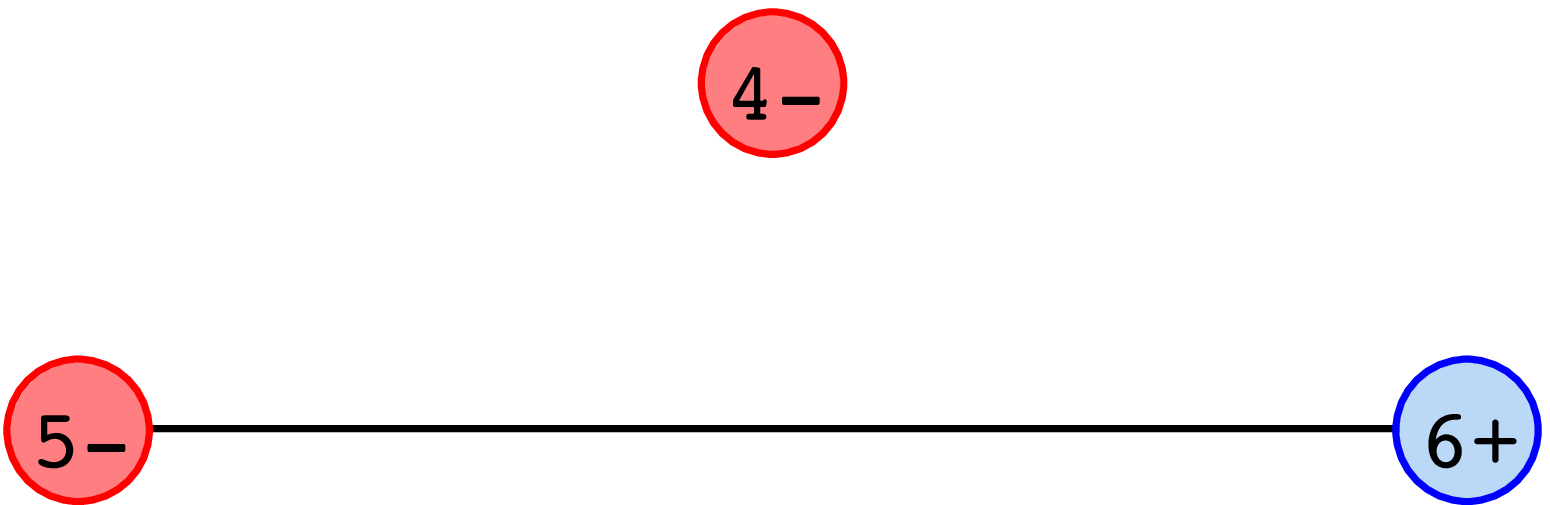}\\
\\ \hline \\
 & $\gnr_{v_4} \circ \gpr_{v_2} \circ \gpr_{v_1} \circ \gpr_{v_3}(G)$\\
& \includegraphics[scale=0.3]{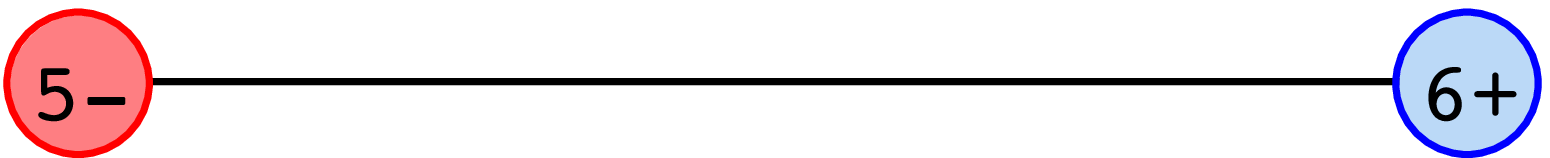}
\end{tabular}
\end{center}
\caption{Two combinatorial reduction strategies removing the same vertices.}\label{pathInvarianceExample}
\label{fig:pathInvarianceExample}
\end{figure}

Among the three combinatorial reduction operations, the negative rule is the only one which is singular (in the sense that it removes a vertex set whose induced subgraph has a singular adjacency matrix). This gives it some special properties, which we shall now study. 

\subsection{Nullity and the negative rule}\label{nullitySection}
Harju et al.\ \cite{harju} ask whether it is true that every combinatorial reduction strategy of a given signed graph applies the negative rule the same number of times, and also whether there is a characterization of those signed graphs that avoid the negative rule in all reductions. We are now able to answer both questions.

\begin{theorem} \label{ngrNullity}
 If $W$ is a reducible vertex set in $G$, then the nullity of $W$ in $G$ is equal to the number of times the negative rule is applied in any combinatorial reduction strategy removing the vertices of $W$.
\end{theorem}
\begin{proof}
 By Corollary \ref{rankCorollary} and the definition of nullity, a reduction $\Gamma_W$ reduces the nullity of $G$ by exactly the nullity of the vertex set $W$. If $W$ is a single negative vertex, then $W$ has nullity $1$, and if $W$ is a positive vertex or a pair of adjacent negative vertices, then $W$ has nullity $0$. The result now follows by induction on the number of reduction rules in the combinatorial reduction strategy.
\end{proof}
\begin{corollary}
 Any combinatorial reduction strategy removing the vertices $W$ uses the same number of negative rules.
\end{corollary}

Recall that all set inclusions in the reducibility poset $\cR(G)$ correspond to reductions of graphs obtainable from $G$ by graph reduction. Theorem \ref{ngrNullity} now shows that for all such inclusions $S \subset T$, where $S \in \cR_m(G)$ and $T \in \cR_n(G)$, $m \leq n$ and the number of negative rules used in any combinatorial reduction strategy realizing this reduction is $n-m$.

Note that we can also give a combinatorial characterization of those signed graphs that avoid $\gnr$. The statement that $A$ is nonsingular means that it has trivial nullspace. Since any vector in $\textbf{F}_2^n$ can be interpreted as a subset of the vertices of the graph, we could state the singularity of $A$ as follows: \textit{$G$ avoids $\emph{gnr}$ if and only if there is no nonempty subset $S$ of $V$ such that each vertex $v \in V$ is adjacent to an even number of vertices in $w \in S$}.

Ehrenfeucht et al.\ \cite{ehrenfeucht} give a similar result for the string formulation of gene assembly. They characterize strings which avoid the string negative rule in all of their reduction as those which have no cycles (see their paper for definitions). Our characterization is related in the sense that cycles in strings correspond to elements of the null space of the adjacency matrix.

If we restrict ourselves to negative graphs (graphs without positive vertices), which can be understood as ordinary undirected graphs with $0$ on the diagonal of their adjacency matrices, this theorem gives a combinatorial interpretation of the binary rank of a graph. This characterization is identical to one given by Godsil and Royle \cite{ggt}, who consider \textit{rank-two reductions} that are identical to $\gdr$.

\section{Pivots and retrographs}\label{pivots}
We now consider the relation between graph reductions as defined above and the pivot operation on simple graphs with loops. The pivot operation is a combinatorial operation on graphs which does not remove any vertices and preserves the binary rank of the graph. The pivot of a matrix is defined by Geelen \cite{geelen}, who considers pivots of matrices over both $\textbf{F}_2$ and $\textbf{R}$. The properties we study here are closely related to results found by Brijder, Harju, and Hoogeboom \cite{brijder}, who also observe the connection between pivots and nonsingular reductions of signed graphs.

In this section, we demonstrate that pivot operations are easy to understand in terms of the pivotal poset of a graph, and use this to derive the relation between pivots and graph reductions. We also consider a special case of the pivot operation, which we call the retrograph of a graph, which is defined for nonsingular graphs and has an interesting combinatorial relation to the original graph.

As has previously been observed in \cite{brijder}, pivots are well-suited to studying nonsingular reductions. It is more difficult to understand their effect on the reducibility poset and graph reductions in general; we shall illustrate this point with an example at the end of the section.

We observe that when pivots are considered for matrices over fields of characteristic other than $2$, slight modifications are needed. In effect, the notation is simplified by the fact that addition and subtraction are identical. However, the modifications necessary to generalize our results to all fields are not difficult.

\subsection{Pivots}
If a matrix $A$ has block form $\smat{P}{Q}{R}{S}$, and $X$ is the set of the first $m$ basis vectors (so that $A$ restricted to columns and rows corresponding to $X$ is $P$), then the \emph{pivot} $A \ast X$ is

\begin{equation}
 A \ast X = \mat{-P^{-1}}{P^{-1}Q}{RP^{-1}}{S - RP^{-1}Q}.
\end{equation}

Some sources give a different definition, differing from ours in the signs of the first block column. The definition we use is the definition from \cite{brijder}, which has the virtue of preserving symmetry of a matrix. Both definitions are, of course, equivalent over $\textbf{F}_2$. Notice that if $A$ is the adjacency matrix of $G$ over $\textbf{F}_2$, and $X$ a subset of vertices, then $A \ast X$ gives the adjacency matrix of a graph with contains the reduction $\Gamma_X(G)$ as a subgraph. This observation is made in \ \cite{brijder}, where several formulas are given expressing the entries of the matrix $A \ast X$ in terms of determinants of submatrices of $A$, and a path invariance result is derived. We shall slightly generalize these formulas to formulas for ranks of submatrices in the pivot. In order to do this, we introduce a useful notation, which makes many properties of the pivot immediately visible: the notion of a pair-class of matrices.

\begin{definition}
 Two pairs of matrices $(A_1,B_1),(A_2,B_2)$ are \emph{row-equivalent} if there exists an invertible matrix $M$ such that $(MA_1, MB_1) = (A_2,B_2)$. Denote by $[A,B]$ the equivalence class of pairs of matrices row-equivalent to $(A,B)$. These equivalence classes are called \emph{pair-classes} of matrices. A pair-class $[A,B]$ is called \emph{proper} if $A$ is nonsingular.
\end{definition}

We observe that, as long as the block matrix $\row{A}{B}$ is full rank, the pair-class $[A,B]$ may be identified with its row space. In this setting, we can define the pair-class as a $1$-dimensional subspace in the $n^{th}$ wedge power of a $2n$-dimensional vector space, or as an element of a Grassmannian variety. Proper pair-classes $[I,A]$ then form a distinguished affine open set of this variety, and determinants of submatrices of $A$ appear as the coordinates in the Pl\"{u}cker embedding of the Grassmanian. Details of these notions may be found in \cite{harris}, chapter 6. This perspective is what originally motivated this approach, and may provide geometric intuition, but we shall not discuss it further because it is not needed for this paper.

Any proper class $[A,B]$ may be written equivalently as $[I,A^{-1}B]$. Indeed, $(I,A^{-1}B)$ is the unique pair in the class $[A,B]$ with $I$ in the first entry, and such a pair exists if and only if the class $[A,B]$ is proper. Thus we have a bijection between square $n \times n$ matrices and proper classes of pairs of $n \times n$ matrices: to a matrix $A$ we may associate the pair-class $[I,A]$, and to a proper pair-class $[A,B]$ we may associate the matrix $A^{-1}B$. We shall define a pivot operation on pair-classes, and show that it is identical, via this correspondence, to the pivot operation as defined for matrices.

\begin{definition}
 Suppose $A = \smat{P_1}{Q_1}{R_1}{S_1}$ and $B = \smat{P_2}{Q_2}{R_2}{S_2}$ are two matrices in block form, where $P_1,P_2$ is an $m \times m$ block and $X$ is the set of the first $m$ basis vectors. The \emph{pivot} of $[A,B]$ by $X$ is
\begin{equation}
 [A, B] \ast X = \left[ \smat{P_2}{Q_1}{R_2}{S_1}, \smat{-P_1}{Q_2}{-R_1}{
S_2} \right]
\end{equation}
\end{definition}

Thus the pivot simply exchanges the first $m$ columns of the two matrices and inverts the signs of these columns in the second matrix. Of course, in $\textbf{F}_2$, the signs are irrelevant, so we may simply view the pivot as exchanging blocks of columns. The following fact is immediate from this definition. Here $X \oplus Y$ denotes the symmetric set difference, $(X \cup Y) \backslash (X \cap Y)$.

\begin{lemma}\label{pivotPathInvariance}
 For a pair-class over $\textbf{F}_2$,
\begin{equation}
 ([A,B] \ast X) \ast Y = [A,B] \ast (X \oplus Y).
\end{equation}
\end{lemma}

Now suppose $A$ is a matrix with block form $\smat{P}{Q}{R}{S}$. Then $A$ corresponds to the pair-class $[I,A]$. Pivoting by the first $m$ basis vectors, we obtain:

\begin{eqnarray*}
 [I,A] \ast X &=& \left[ \smat{P}{0}{R}{I}, \smat{-I}{Q}{0}{S} \right],
\end{eqnarray*}

which is proper if and only if $P$ is invertible. In this case, we can rewrite it by performing row operations as follows:

\begin{eqnarray*}
 [I,A] \ast X &=& \left[ \smat{I}{0}{R}{I}, \smat{-P^{-1}}{P^{-1}Q}{0}{S} \right]\\
&=& \left[ \smat{I}{0}{0}{I}, \smat{-P^{-1}}{P^{-1}Q}{R P^{-1}}{S - R P^{-1} Q} \right]\\
&=& [I, A \ast X].
\end{eqnarray*} 

Thus we see that pivots between proper pair-classes correspond precisely to the matrix pivot operation defined at the beginning. The main benefit of considering pivots on pair-classes rather than matrices is the ease with which we may consider ranks of submatrices in this framework. We also observe that this correspondence, combined with Lemma \ref{pivotPathInvariance}, gives an easy proof of the fact that for matrices over $\textbf{F}_2$, $(A \ast X) \ast Y = A \ast (X \oplus Y)$. In particular, this shows path invariances for sequences of pivot operations on disjoint vertices. This path invariance property is also observed in \cite{brijder}, in generalizing similar theorems from \cite{arratia}, \cite{genest}, and \cite{oum}. These papers approach the problem by analyzing determinants of submatrices of pivots, a computation which will follow (in $\textbf{F}_2$) as a special case of formulas we observe below for ranks of submatrices of pivots.

Another reason we have chosen to approach pivots from the perspective of pair-classes is the ease with which we can study the pivotal poset in this context.

\begin{definition}
 Suppose $[A,B]$ is a pair-class of matrices, with rows and columns indexed by $V$, and $W_1, W_2 \subset V$ are two subsets such that $|W_1| = |W_2|$. Then $N_{W_1,W_2}([A,B])$ denotes the nullity (i.e. dimension of the kernel) of the matrix formed by the columns $V \backslash W_1$ from $A$ and the columns $W_2$ from $B$.
\end{definition}

\begin{definition}
 The \emph{pivotal poset} of a pair-class $[A,B]$, denoted $\cR_0([A,B])$, is the set of subsets $W \subset V$ such that $N_{W,W}([A,B]) = 0$.
\end{definition}

It is not difficult to verify that $N_{W_1,W_2}([A,B])$ and $\cR_0([A,B])$ are well-defined, in the sense that they do not depend on which representative of the equivalence class $[A,B]$ is chosen. We have chosen these definitions due to their meaning in case of proper pair-classes.

\begin{lemma}
 For a proper pair-class $[I,A]$, $N_{W_1,W_2}([I,A])$ is the nullity of the submatrix of $A$ on rows $W_1$ and columns $W_2$. If $A$ is the adjacency matrix of a graph $G$, then $\cR_0([I,A]) = \cR_0(G)$. A pair-class $[A,B]$ is proper if and only if $\emptyset \in \cR_0([A,B])$.
\end{lemma}
\begin{proof}
 The first statement follows easily by writing $I$ and $A$ in block form. The second statement follows immediately from the first and the definition of a proper pair-class.
\end{proof}

The following lemma can in fact be used to uniquely characterize the pivot operation over the field $\textbf{F}_2$. The corollary following the lemma will allow us to uniquely characterize the pivot operation on graphs. In Corollary \ref{pivotPosets} and elsewhere, we write $S \oplus X$, which $S$ is a set of sets of vertices and $X$ is a set of vertices, to indicate the set $\{Y \oplus X: Y \in S \}$.

\begin{lemma}\label{pivotNullity}
 If $[A,B]$ is a pair-class, with rows and columns of $A,B$ indexed by $V$, then for any sets $W_1,W_2,X \subset V$ such that $|W_1| = |W_2|$,
\begin{equation}\label{pivotLemmaEq1}
 N_{W_1,W_2} ([A,B] \ast X) = N_{(W_1 \cap X^c) \cup (W_2^c \cap X), (W_2 \cap X^c) \cup (W_1^c \cap X)}([A,B]),
\end{equation}
where the superscript $c$ indicates complement in $V$.
\end{lemma}
\begin{proof}
 The two sides of equation \ref{pivotLemmaEq1} denote the nullities of matrices which are identical up to permutation and signing of the columns, which thus have the same rank.
\end{proof}
\begin{corollary}\label{pivotPosets}
 For any pair-class $[A,B]$ and subset $X \subset V$,
\begin{equation}
 \cR_0([A,B] \ast X) = \cR_0([A,B]) \oplus X.
\end{equation}
\end{corollary}
\begin{proof}
 This follows by considering the special case $W_1 = W_2$ in Lemma \ref{pivotNullity}, which is $N_{W,W}([A,B]\ast X) = N_{W \oplus X, W \oplus X}([A,B])$. 
\end{proof}

We observe that, over $\textbf{F}_2$, Lemma \ref{pivotNullity} may be regarded as a generalization of the following determinant formula, discussed in \cite{brijder} (proposition 3).

\begin{equation}
 \textrm{det}(A \ast X)_{Y,Y} = \pm \textrm{det} A_{X \oplus Y} / \textrm{det}(A_{Y,Y}).
\end{equation}

In fact, our method of pair-classes can also be used to establish this result over general fields without much effort, although the definition must be modified to consider two pairs equivalent only if they differ by multiplication on the left by a matrix of determinant $1$, rather than any invertible matrix. Over $\textbf{F}_2$, this distinction is nonexistant.

We have now established the main properties of pivots of pair-classes, which allow us to characterize the pivot of a graph combinatorially.

\begin{theorem}
 If $G$ is a graph with pivotal poset $\cR_0(G)$ and $W\subset V$ is a subset of the vertices of $G$, then $\cR_0(G) \oplus W$ is realizable if and only if $W \in \cR_0(G)$. If $G$ has adjacency matrix $A$, then the graph realizing $\cR_0(G) \oplus W$ has adjacency matrix $A \ast W$.
\end{theorem}
\begin{proof}
 Suppose that $\cR_0(G) \oplus W$ is realizable by the graph $H$ with adjacency matrix $B$. Then $\emptyset \in \cR_0(H) = \cR_0(G)$, hence $\emptyset \oplus W = W \in \cR_0(G)$. Then observe that the pair-class $[I,B]$ must have the same pivotal poset as $[I,A] \ast W$, by Corollary \ref{pivotPosets}. Since $[I,A] \ast X = [I,A \ast X]$, we must have $B = A \ast X$.

Conversely, suppose that $W \in \cR_0(G)$. Then the pair-class $[I, A \ast W] = [I,A] \ast W$ has pivotal poset $\cR_0(G) \oplus W$, by a similar analysis to above. Since $A \ast W$ is symmetric, it is the adjacency matrix of a graph realizing the pivotal poset $\cR_0(G) \oplus W$, as desired.
\end{proof}

\begin{definition} \label{pivotDefinition}
 If $G$ is a graph, and $W \in \cR_0(G)$, then the graph whose pivotal poset is $\cR_0(G) \oplus W$ is called the \emph{pivot of $G$ by $W$} and is denoted $P_W(G)$.
\end{definition}

Observe that, given this characterization of the pivot of a graph, the following path invariance property is immediately clear. This is essentially the pivot analogue of Theorem \ref{pathInvariance}.

\begin{theorem}
 If $G$ is a graph, and $W_1,W_2 \subset V$ are two sets of vertices with $W_1 \in \cR_0(G)$, then $W_2 \in \cR_0(P_{W_1}(G))$ if and only if $W_1 \oplus W_2 \in \cR_0(G)$, and $P_{W_2} \circ P_{W_1} (G) = P_{W_1 \oplus W_2}(G)$.
\end{theorem}
\begin{proof}
 The first assertion is a consequence of definition \ref{pivotDefinition}. The second follows since the pivotal posets of $P_{W_2} \circ P_{W_1} (G)$ and $P_{W_1 \oplus W_2}(G)$ are both $\cR_0{G} \oplus W_1 \oplus W_2$.
\end{proof}

As we remarked at the beginning of this section, there is a close relation between pivots of graphs and graph reductions. This relation is expressed in the following theorem.

\begin{theorem}\label{pivotReduction}
 If $G$ is a signed graph, and $W_1,W_2 \subset V$ are two sets of vertices such that $W_1$ and $W_1 \backslash W_2$ both lie in $\cR_0(G)$, then
\begin{equation}
 I_{W_2} \circ P_{W_1} (G) = P_{W_1 \cap W_2} \circ \Gamma_{W_1 \backslash W_2} \circ I_{W_1 \cup W_2} (G),
\end{equation}
where $I_{U}(G)$ denotes the induced subgraph on vertices $U$ of $G$.
\end{theorem}
\begin{proof}
 First we show that the expression on the right side of the equation is well-defined. $\Gamma_{W_1 \backslash W_2}$ applies to $I_{W_1 \cup W_2}(G)$ because we have assumed that $W_1 \backslash W_2$ is nonsingular in $G$. Now $P_{W_1 \cap W_2}$ applies to the graph $\Gamma_{W_1 \backslash W_2} \circ I_{W_1 \cup W_2} (G)$ if and only if $W_1 \cap W_2$ is nonsingular in $\Gamma_{W_1 \backslash W_2} \circ I_{W_1 \cup W_2} (G)$, which is true if and only if $(W_1 \cap W_2) \cup (W_1 \backslash W_2) = W_1$ is nonsingular in $I_{W_1 \cup W_2}(G)$, which follows from our assumptions.

Now the graphs described by the two sides of this equation have the same vertex set, so by Theorem \ref{posetDeterminesGraph}, it suffices to show that they have the same pivotal poset. Now observe that
\begin{eqnarray*}
 && \cP_0(I_{W_2} \circ P_{W_1} (G))\\ &=& \{ S \subset W_2: S \oplus W_1 \in \cP_0(I_{W_1 \cup W_2}(G)) \}\\
&=& \{ S \subset W_2: (S \oplus (W_1 \cap W_2)) \cup (W_1 \backslash W_2) \in \cP_0(I_{W_1 \cup W_2} (G)) \}\\
&=& \{ S \subset W_2: S \oplus (W_1 \cap W_2) \in \cP_0(\Gamma_{W_1 \backslash W_2} \circ I_{W_1 \cup W_2} (G)) \}\\
&=& \cP_0(\Gamma_{W_1 \backslash W_2} \circ I_{W_1 \cup W_2} (G)) \oplus (W_1 \cap W_2)\\
&=& \cP_0(P_{W_1 \cap W_2} \circ \Gamma_{W_1 \backslash W_2} \circ I_{W_1 \cup W_2} (G)).
\end{eqnarray*}
Thus these two graphs have the same pivotal poset, and thus are equal.
\end{proof}

There are two important special cases of Theorem \ref{pivotReduction}, expressed in the following corollary.

\begin{corollary}
 If $U \cup W = V$, and $U, U \backslash W \in \cR_0(G)$, then
\begin{equation}\label{pivotsOfReductions}
 I_W \circ P_U (G) = P_{U \cap W} \circ \Gamma_{U \backslash W} (G).
\end{equation}
If $U$ and $W$ are disjoint and $U \in \cR_0(G)$, then
\begin{equation}\label{pivotsHaveReductions}
 I_W \circ P_U (G) = \Gamma_U (G).
\end{equation}
\end{corollary}

Equation \ref{pivotsHaveReductions} simply expresses the fact (already evident from the adjacency matrix) that, at least in the case of nonsingular reductions, we can find any graph reduction as an induced subgraph of a pivot. Equation \ref{pivotsOfReductions} demonstrates that all pivots of a reduction of a graph can be obtained by simply finding pivots of the original graph and examining an induced subgraph. A combinatorially interesting special case of this is studied in the next section.

Before concluding this section, we remark that while the pivotal poset of a graph is very well behaved under pivots, the reducibility poset is not. In other words, it is easy to characterize the nonsingular combinatorial reduction strategies of the pivot of a graph, but it is harder to characterize the stages at which the negative rule $\gnr$ will apply. For example, consider the adjacency matrices in Figure \ref{fig:posets}. Arrows indicate set inclusions, and subscripts indicate the nullity of the vertex set. The pivotal poset is simply the sub-poset of those sets with subscript $0$.

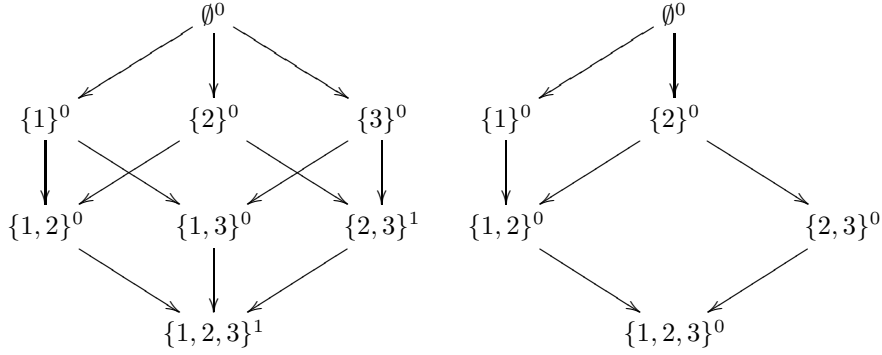
\begin{figure}
\begin{center}
\begin{tabular}{cc}
  $A = \left( \begin{array}{ccc} 1&0&0 \\ 0&1&1 \\ 0&1&1 \end{array} \right)$
&
$A \ast \{1,2\} = \left( \begin{array}{ccc} 1&0&0 \\ 0&1&1 \\ 0&1&0 \end{array} \right)$
\vspace{1cm}
\\ 
\xymatrix{
	& \emptyset^0 \ar[dl] \ar[d] \ar[dr] \\
	\{1\}^0 \ar[d] \ar[dr] & \{2\}^0 \ar[dl] \ar[dr] & \{3\}^0 \ar[dl] \ar[d] \\
	\{1,2\}^0 \ar[dr] & \{1,3\}^0 \ar[d] & \{2,3\}^1 \ar[dl] \\
	& \{1,2,3\}^1 
}
&
\xymatrix{
	& \emptyset^0 \ar[dl] \ar[d] \\
	\{1\}^0 \ar[d] & \{2\}^0 \ar[dl] \ar[dr] \\
	\{1,2\}^0 \ar[dr] && \{2,3\}^0 \ar[dl] \\
	& \{1,2,3\}^0
}
\end{tabular}
\end{center}
\caption{Two adjacency matrices and their reducibility posets.}
\label{fig:posets}
\end{figure}

Let $G_1$ be the graph corresponding to the matrix $A$ in Figure \ref{fig:posets}, and $G_2$ be the graph corresponding to $A \ast \{1,2\}$. Note that of course the pivotal poset of $G_2$ is obtained by taking the symmetric set difference of each pivotal set in $G_1$ with $\{1,2\}$. However, $G_2$ is nonsingular, whereas $G_1$ has nullity $1$, thus the reducibility poset of $G_1$ includes singular sets (namely $\{2,3\}$ and $\{1,2,3\}$, which are both in $\cR_1(G_1)$. This example shows that there cannot in general be a bijection between the reducibility poset of a graph and the reducibility poset of its pivot, since these posets may have different sizes.

\subsection{The retrograph}
If $G$ is a nonsingular graph (that is, its entire vertex set $V$ is pivotal, which is to say that no combinatorial reductions strategies of $G$ use the negative rule), then we can consider the pivot of $G$ by its entire vertex set. The graph obtained in this way has interesting combinatorial properties, which we shall now consider.

\begin{definition}
 If $G$ is a graph with vertex set $V$, and $V \in \cR_0(G)$, then the \emph{retrograph} $G^R$ of $G$ is the pivot by the entire vertex set, $P_V(G)$.
\end{definition}

The retrograph has two useful combinatorial properties, expressed by the following theorem.

\begin{theorem}\label{retrographTheorem}
 If $G$ has retrograph $G^R$, then any successful combinatorial reduction strategy of $G$ applies in reverse to $G^R$. The retrograph $G^R$ is the unique graph with whose combinatorial reduction strategies are the same as those for $G$ in reverse. If $W \subset V$ is a subset of the vertices of $G$, then $W$ is reducible in $G$ if and only if $V \backslash W$ is reducible in $G^R$, and in this case $(\Gamma_W (G))^R = I_{V \backslash W} (G^R)$.
\end{theorem}
\begin{proof}
 The statement that a successful combinatorial reduction strategy of $G$ applies in reverse to $G^R$ is equivalent to the statement that a subset of the vertices of $G$ is reducible in $G$ if and only if its complement is reducible in $G^R$. This latter statement is equivalent to $\cR_0(G^R) = \cR_0(G) \oplus V$, which is true by the definition of the pivot of a graph. The statement $(\Gamma_W (G))^R = I_{V \backslash W} (G^R)$ follows from equation \ref{pivotsOfReductions} by letting the sets $W_1,W_2$ in the statement of equation \ref{pivotsOfReductions} be the sets $V,V\backslash W$, respectively.
\end{proof}

Observe that the combinatorial description of $G^R$ (the graph whose successful combinatorial reduction strategies are the combinatorial reductions strategies of $G$ applied in reverse) cannot be used to describe an analogous retrograph for singular graphs $G$. This is because $\cR_0(G^R) = \cR_0(G) \oplus V$ implies that, since $\emptyset$ is pivotal in any graph, the whole vertex set $V$ must be pivotal in $G$ in order for $G^R$ to be well-defined. Thus the retrograph, defined by this combinatorial property, exists if and only if $G$ is nonsingular.

The first part of this theorem shows the first use of the retrograph: it allows us to look ahead and immediately see how a graph reduction strategy must end, without actually computing the entire reduction strategy. The second part of the theorem shows that, in some sense, the retorgraph reduces the study of reductions of a given graph to the study of subgraphs of the retrograph. More precisely, if we wish to verify some statement on the result of every reduction of $G$, and we can formulate the statement in such a way that it is easy to verify on the retrograph, then we can simply verify this latter statement on all subgraphs of the retrograph, and avoid computing any reductions.

We also observe that Theorem \ref{retrographTheorem} can be restated in terms of the adjacency matrix to obtain an interesting matrix identity. In fact, this identity holds for matrices over any field, by considering pair-classes over fields other than $\textbf{F}_2$. It may also be proved directly by algebra.

\begin{corollary}
 If $A$ is an invertible $n \times n$ matrix, $V$ is the set $1,2,\dots,n$, regarded as both the set of rows and the set of columns, and $X \subset V$, then $A\langle X,X \rangle$ is invertible if and only if $A^{-1} \langle V \backslash X, V \backslash X \rangle$ is invertible. In this case, if $A$ is written in block form as $\smat{P}{Q}{R}{S}$ with $A \langle X,X \rangle = P$, then the matrix $A^{-1} \langle V \backslash X, V \backslash X \rangle$ is equal to $(S - QP^{-1}R)^{-1}$.
\end{corollary}

We illustrate the retrograph concept with a simple example. Consider the graph and retrograph shown in Figure \ref{fig:retrograph}.

\begin{figure}[h]
\begin{center}
\begin{tabular}{c||c}
$A = \left( \begin{array}{ccccc} 1&0&0&1&1\\0&1&1&1&1\\0&1&1&0&0\\1&1&0&1&1\\1&1&0&1&0 \end{array} \right)$
& $A^{-T} = \left( \begin{array}{ccccc}1&1&1&0&0\\1&0&0&1&0\\1&0&1&1&0\\0&1&1&1&1\\0&0&0&1&1 \end{array} \right)$\\
\\ \hline \\
\includegraphics[scale=0.3]{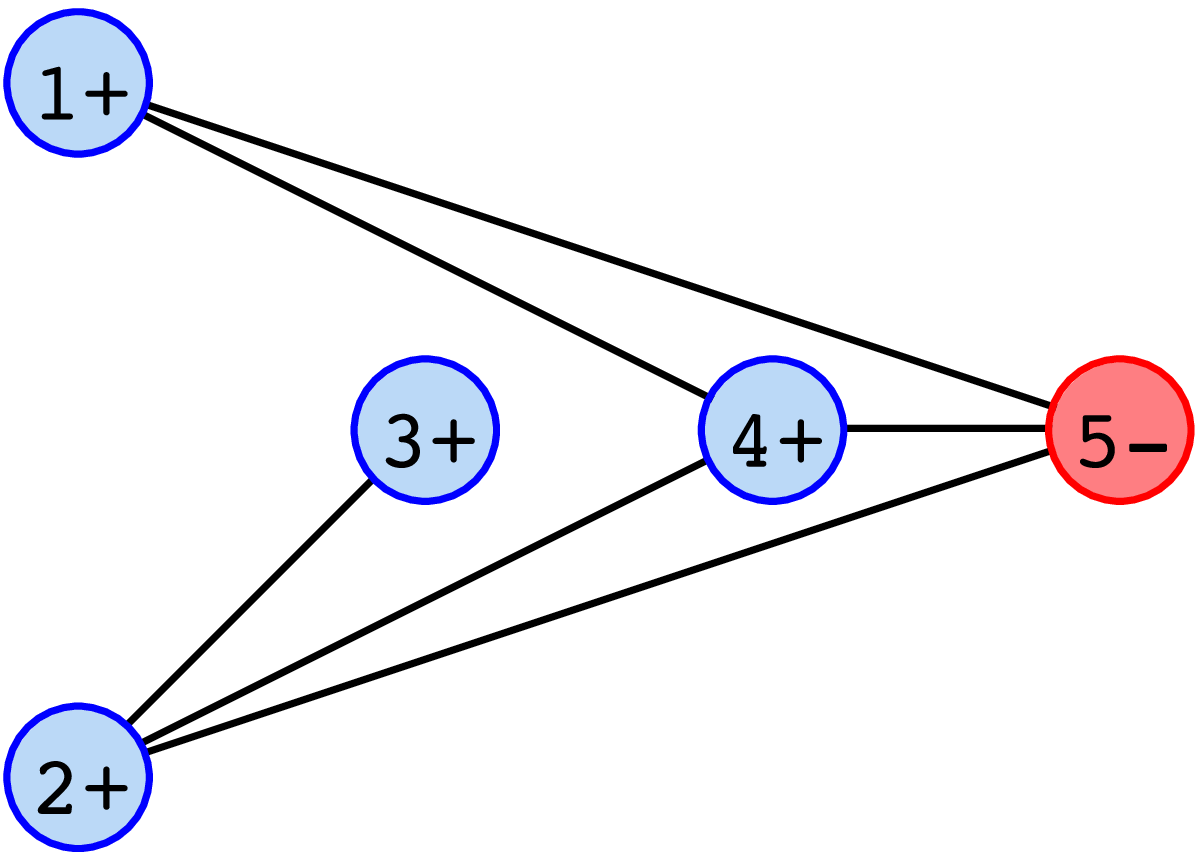} & \includegraphics[scale=0.3]{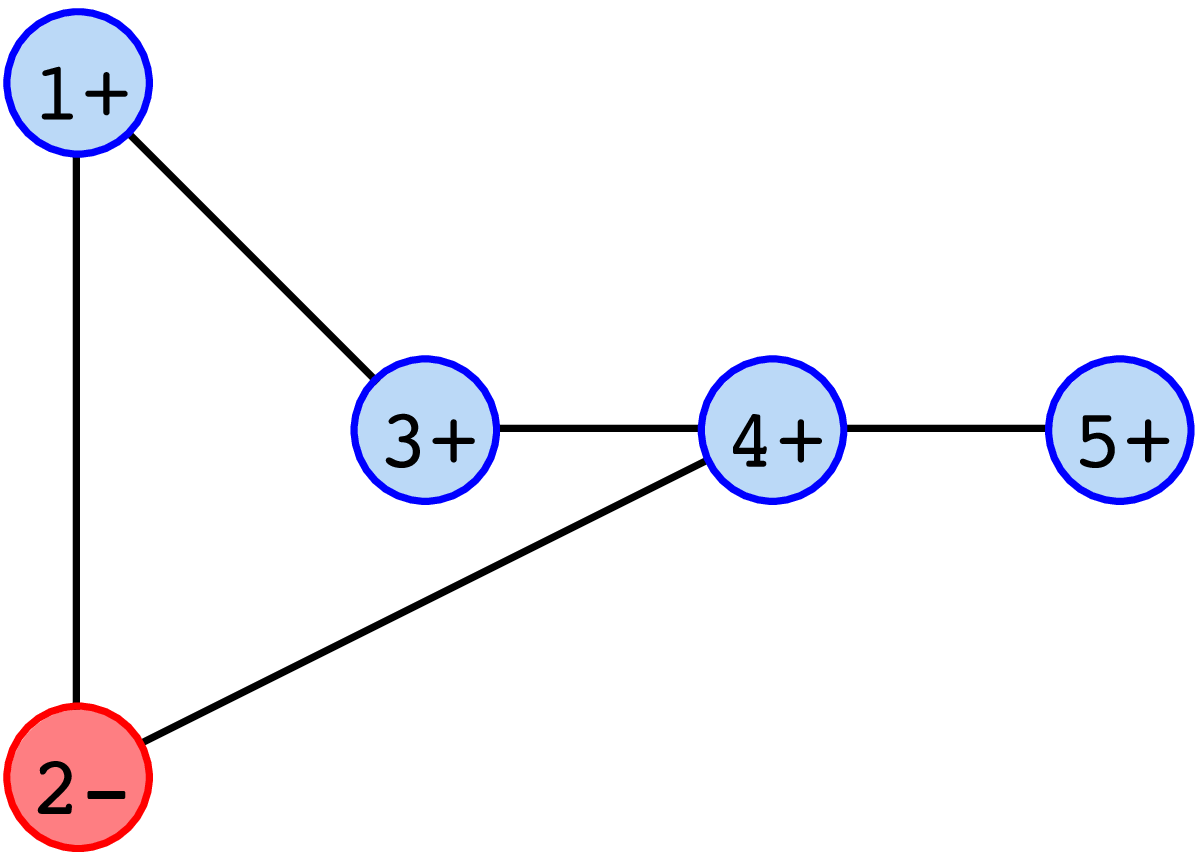}\\
\end{tabular}
\end{center}
\caption{A graph $G$ and its retrograph $G^R$.}
\label{fig:retrograph}
\end{figure}

We show in Figure \ref{fig:retroReduction} one successful reduction strategy of this graph, and show the retrograph of the result at each stage. Observe that, in keeping with Theorem \ref{retrographTheorem}, these retrographs of reductions are simply induced subgraphs of the original retrograph.

\begin{figure}[h]
\begin{center}
\begin{tabular}{c||c}
$\gpr_{v_1}(G)$ & $I_{\{v_2,v_3,v_4,v_5\}}(G^R)$\\&\\
\includegraphics[scale=0.3]{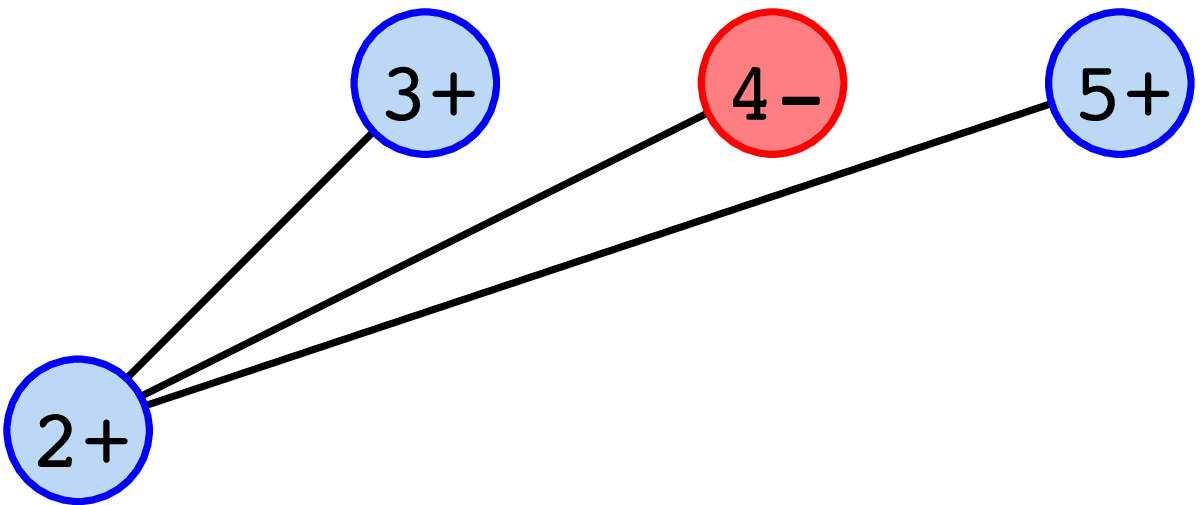} & \includegraphics[scale=0.3]{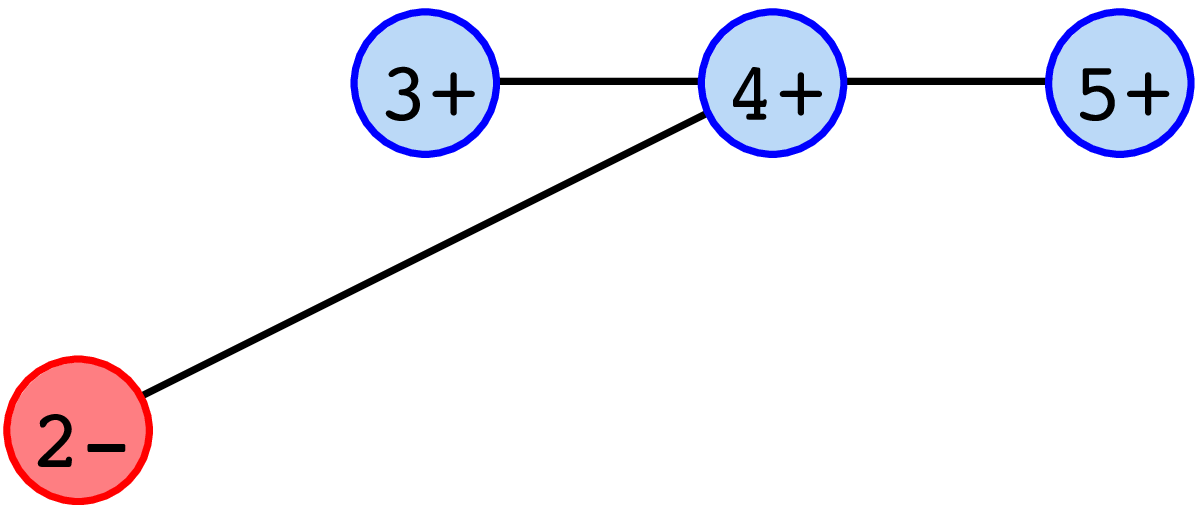}\\
\\ \hline \\
$\gpr_{v_5} \circ \gpr_{v_1} (G)$ & $I_{\{v_2,v_3,v_4\}}(G^R)$\\&\\
\includegraphics[scale=0.3]{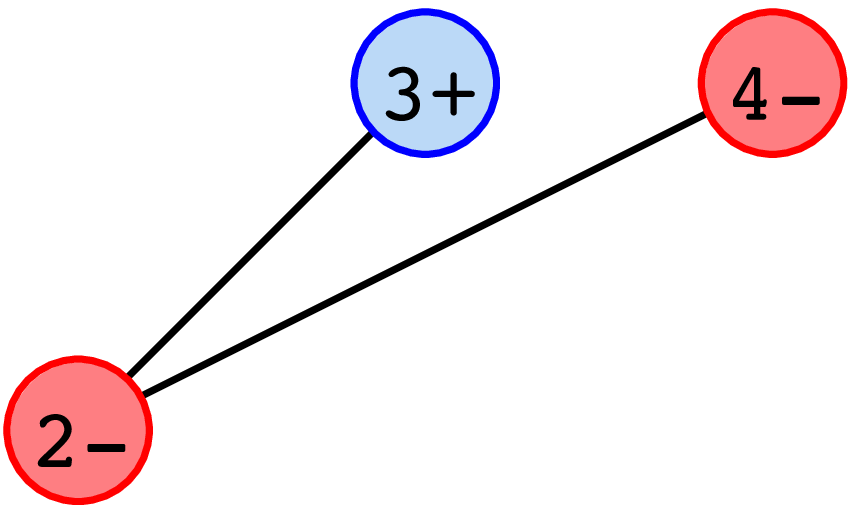} & \includegraphics[scale=0.3]{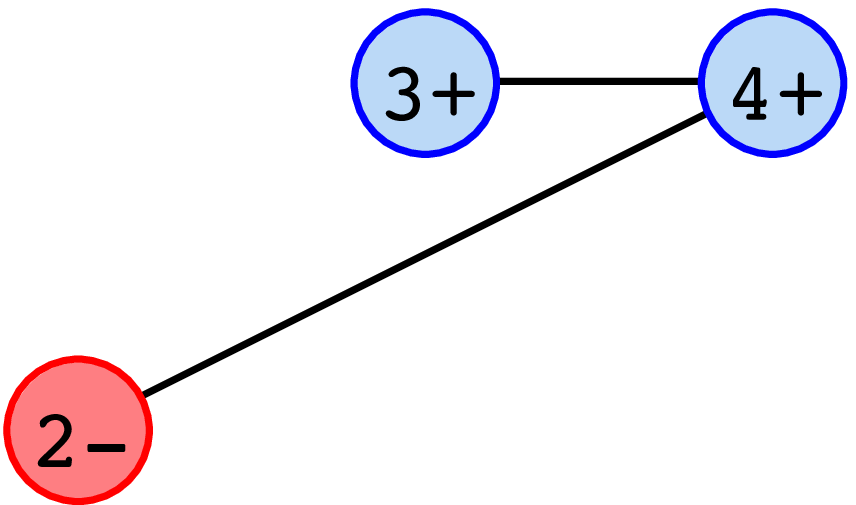}\\
\\ \hline \\
$\gdr_{v_2,v_4} \circ \gpr_{v_5} \circ \gpr_{v_1} (G)$ & $I_{\{v_3\}} (G^R)$\\&\\
\includegraphics[scale=0.3]{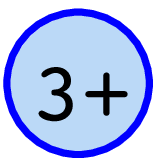} & \includegraphics[scale=0.3]{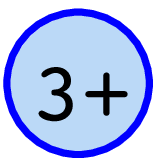}\\
\end{tabular}
\end{center}
\caption{A reduction strategy of $G$, with retrographs at each stage.}
\label{fig:retroReduction}
\end{figure}

\subsection{Reverse reductions}\label{reverseReductions}
Graph pivots also provide a simple characterization of what might be called the inverse problem for graph reductions: given a graph $G$, which graphs $G'$ can be transformed into $G$ by a graph reduction?

Consider first the case where $G'$ can be transformed to $G$ by a nonsingular reduction. If $V$ are the vertices of $G$, and $W$ the other vertices of $G'$, we can express this by writing $\Gamma_W(G') = G$. By Theorem \ref{pivotReduction}, this is equivalent to $I_V \circ P_W (G') = G$. Denoting by $H$ the graph $P_W(G')$, and recalling that $P_W(H) = G'$, we have the following bijection (once we fix a set $W$ disjoint from $V$).

\begin{equation}\label{reverse1}
 \{G':\ \Gamma_W(G') = G \} = \{P_W(H):\ I_V(H) = G\ \textrm{and}\ W \in \cR_0(H) \}
\end{equation}

On the left side of equation \ref{reverse1}, $H$ ranges over all graphs on vertices $V \cup W$.

In general, $G'$ can be reduced to $G$ if and only if there is a nonsingular reduction from $G'$ to a union of $G$ with some number of isolated negative vertices.

\begin{theorem}
 If $G$ is any graph with vertices $V$, and $W$ is a set of vertices disjoint from $V$, then the set of graphs $G'$ on vertices $V \cup W$ such that $\Gamma_W (G') = G$ is precisely $\{ P_{W_1} (H)\}$, where $W_1$ ranges over all subsets of $W$ and $H$ ranges over all graphs on vertices $V \cup W$ such that $W_1 \in \cR_0(H)$ and all vertices in $W \backslash W_1$ do not have loops and are not adjacent to any vertices in $V$.
\end{theorem}
\begin{proof}
 Suppose $G' = P_{W_1} (H)$ satisfies the given conditions. Then clearly $W \backslash W_1$ is reducible in $\Gamma_{W_1}(G')$, since all these vertices are negative and isolated. Thus by Theorem \ref{pathInvariance}, $W$ is reducible in $G'$, and $\Gamma_W(G') = G$, since there exists a combinatorial reduction strategy from $G'$ to $\Gamma_{W_1}(G')$, and the remaining vertices $W \backslash W_1$ can be removed by the negative rule to obtain $G$, which is then equal to $\Gamma_W(G')$.

Conversely, suppose $G'$ is a graph on vertices $V \cup W$ such that $W \in \cR(G')$ and $\Gamma_W (G') = G$. Then if $W_1$ is a maximal subset of $W$ nonsingular in $G'$, $\textrm{rank}_{G'}(W_1)$ must be $\textrm{rank}_{G'}(W)$, hence reducing $W_1$ results in the disjoint union of $G$ and a collection of isolated negative vertices. Hence $G'$ has the desired form.
\end{proof}

\section{Conclusion}
In this paper, we have introduced new algebraic methods for the study of the graph formalization of gene assembly, demonstrating in particular the close relation between combinatorial graph reductions and linear algebra over $\textbf{F}_2$, and giving combinatorial interpretations to the binary rank of a graph and the inverse of its adjacency matrix. Our general definition of reducibility and graph reduction, together with the path invariance property, shed light on the relation between the three combinatorial reduction rules. Some of this relation had been uncovered in \cite{brijder}, although our methods successfully incorporate the negative rule into the analysis and approach the problem using different methods. We have also generalized results from \cite{brijder} on the pivot operation and its relation to graph reductions, particularly the special case which we call the retrograph. We believe in particular that our approach of considering matrix pivots by means of pair-classes of matrices may shed considerable light on the properties of pivots of graphs. Finally, the reducibility poset and pivotal poset give a new and useful way to phrase many problems about graph reductions, in particular regarding the study of parallel complexity.

The most mysterious aspect of our method in this paper is the relation between the pivotal poset $\cR_0(G)$ and the other levels $\cR_n(G)$ of the reducibility poset of a graph. Theorem \ref{posetDeterminesGraph} demonstrates that the pivotal poset completely determines the other levels of the reducibility poset, but the method of recovered the latter from the former is rather cumbersome. In particular, it seems difficult to understand the effect of pivot operations on the full reducibility poset. It is possible that some restricted class of pivot operations are better behaved in this regard.

A more combinatorial way of stating the difficulties described above is that very little is currently understood about when negative rules may occur in the course of a combinatorial reduction strategy. The original problem which led to this paper was the verification that the \textit{number} of times the negative rule occurs in a successful reduction is a graph invariant, but presumably much more could be said about the places in a reduction strategy that the negative rule could occur. All this essentially amounts to understanding the structure of the full reducibility poset.

Although pair-classes of matrices were extremely convenient in studying pivot operations on graphs, it seems that they are the wrong structure to consider pivots and graph reductions. First, the fact that pivots of symmetric matrices remain symmetric appears to be somewhat coincidental; a more natural formulation of pivots might make this fact obvious. A more intrinsic way of stating this criticism is to observe that the adjacency matrix should be viewed as a symmetric bilinear form, not as a linear transformation, as was made vivid in Section \ref{algebraic}. The correct definition of the pivot operation should more explicitly respect this aspect of the adjacency matrix. If such a definition can be found, it might more naturally subsume the intrinsic definition of graph reductions made in Section \ref{algebraic}, and perhaps illuminate the difficulties mentioned in relating the pivotal poset to the full reducibility poset. The notion of the retrograph, which can easily be defined intrinsically , may be critical to this problem. The pivot operation may be regarded as interpolating between the concepts of graph reduction and the retrograph, since the retrograph is a special case, and graph reductions appear as subgraphs of pivots. Thus a more natural definition of the pivot of a graph would presumably interpolate between our definitions in Section \ref{algebraic} and some intrinsic definition of the retrograph.

We have laid some groundwork for an investigation of parallel complexity using algebraic methods. The following problems, to which these methods might be useful, remain open. These questions can also be stated in terms of the reducibility poset of a graph, or equivalently in terms of ranks of submatrices of the adjacency matrix.

\begin{problem}
Let $f(n)$ denote the largest parallel complexity of a signed graph on at most $n$ vertices. Is $f(n)$ bounded by a constant? If not, what is its asymptotic behavior as $n$ approaches infinity? What if $f(n)$ instead denotes the largest parallel complexity of a graph on $n$ negative vertices?
\end{problem}

It has been conjectured in \cite{harju} that the function $f(n)$ is in fact bounded by a constant. The best known upper bound is linear in $n$.

\begin{problem}
Given a graph $G$ on $2n$ negative vertices, partitioned into $n$ edges $e_1, e_2, \dots, e_n$ on disjoint vertex sets, is there an efficient algorithm to determine whether the $n$ double rules $\textrm{gdr}_{e_i}$ removing these edges apply in parallel?
\end{problem}

Both these questions could also be asked in terms of average behavior. Of course both of the following problems are not currently well-posed, since various probability distributions could be chosen in both cases.

\begin{problem}
What is the average parallel complexity of a signed graph on $n$ vertices?
\end{problem}

\begin{problem}
If $n$ disjoint edges $e_1,\dots e_n$ between negative vertices are fixed, and edges between the vertices of the $e_i$ are either added or not added at random, what is the probability that the $n$ double rules $\textrm{gdr}_{e_i}$ removing these edges apply in parallel?
\end{problem}

\section{Acknowledgments}
This research was done at the University of Minnesota Duluth with the financial support of the National Science Foundation (grant number DMS-0447070-001) and the National Security Agency (grant number H98230-06-1-0013). I gratefully acknowledge the advice and assistance of Nathan Kaplan, and Ricky Liu throughout the project, and Raju Krishnamoorthy for reading this paper and providing comments. Geir Helleloid and Jack Huizenga also provided useful suggestions. I am grateful to the anonymous referee of an earlier draft for making me aware of the matrix pivot operation and its relevance to this topic. Finally, I would like to thank Joe Gallian for introducing me to the topic, and for support throughout the project.

\end{document}